\documentclass[11pt]{amsart}
\usepackage{amsmath,amsthm,amsfonts,amssymb,amscd,flafter,pinlabel}
\usepackage[mathscr]{eucal}
\usepackage{graphics}

\usepackage[usenames,dvipsnames]{color}
\usepackage{subfigure}
\usepackage{graphicx}

\usepackage{pgf}
\usepackage{tikz}
\usetikzlibrary{arrows,automata}
\usepackage[latin1]{inputenc}

\usepackage[colorlinks=true, urlcolor=NavyBlue, linkcolor=NavyBlue, citecolor=NavyBlue, pdftitle={On the equivalence of Legendrian and transverse invariants in knot Floer homology}, pdfauthor={John A. Baldwin, David Shea Vela-Vick, Vera Vertesi}, pdfsubject={Transverse Invariants}, pdfkeywords={Legendrian knots, transverse knots, Heegaard Floer homology, 57M27, 57R58}]{hyperref}

\newtheorem{theorem}{Theorem}[section]
\newtheorem{corollary}[theorem]{Corollary}

\newtheorem{proposition}[theorem]{Proposition}

\newtheorem{lemma}[theorem]{Lemma}

\theoremstyle{definition}

\theoremstyle{definition}
\newtheorem{remark}[theorem]{Remark}

\newcommand\Lm{\mathscr{L}}
\newcommand\Lh{\widehat{\mathscr{L}}}
\newcommand\Gm{\lambda}
\newcommand\Gh{\widehat{\lambda}}

\newcommand\Th{\widehat{\mathscr{T}}}
\newcommand\Tm{\mathscr{T}}

\newcommand\HFK{{\rm {HFK}}}
\newcommand\HFKh{{\rm {\widehat{HFK}}}}

\newcommand\HFKm{{\rm {HFK^{-}}}}
\newcommand\CFK{{\rm {CFK}}}
\newcommand\CFKh{{\rm {\widehat{CFK}}}}
\newcommand\CFKm{{\rm {CFK^{-}}}}

\newcommand\alphas{\mbox{\boldmath$\alpha$}}
\newcommand\betas{\mbox{\boldmath$\beta$}}
\newcommand\Sym{{\rm {Sym}}}
\newcommand\ws{\mathbf w}
\newcommand\zs{\mathbf z}
\newcommand\F{\mathbb F}
\newcommand\M{\mathcal M}
\newcommand\HD{\mathcal H}

\newcommand\coker{\rm coker\,}
\def\x{\mathbf{x}}

\def\y{\mathbf{y}}

\newcommand{\mscr}{\mathscr}

\begin{document}

\title[{On the equivalence of Legendrian and transverse invariants}]{On the equivalence of Legendrian and transverse invariants in knot Floer homology}

\author[John A. Baldwin]{John A. Baldwin}
\address{Department of Mathematics \\ Princeton University}
\email{baldwinj@math.princeton.edu}
\urladdr{\href{http://www.math.princeton.edu/~baldwinj}{http://www.math.princeton.edu/\~{}baldwinj}}

\author[David Shea Vela-Vick]{David Shea Vela--Vick}
\address{Department of Mathematics \\ Columbia University}
\email{shea@math.columbia.edu}
\urladdr{\href{http://www.math.columbia.edu/~shea}{http://www.math.columbia.edu/\~{}shea}}

\author[Vera V\'ertesi]{Vera V\'ertesi}
\address{Department of Mathematics \\ Massachusetts Institute of Technology}
\email{vertesi@math.mit.edu}
\urladdr{\href{http://www.math.mit.edu/~vertesi}{http://www.math.mit.edu/\~{}vertesi}}

\thanks{JAB was partially supported by a NSF Postdoctoral Fellowship DMS-0802975 and NSF Grant DMS-1104688.\\
\indent DSV was partially supported by a NSF Postdoctoral Fellowship DMS-0902924.\\
\indent VV was partially supported by NSF Grant DMS-1104690}

% \date{\today}
\maketitle

%%%%%%%%%%%%%%%%%%%%%%%%%%%%%%%%%%%%%%%%%%%%%%%%%%%%%%%

%%%%%%%%%%%%%%%%%%%%%%%%%%%%%%%%%%%%%%%%%%%%%%%%%%%%%%%
\begin{abstract}
%%%%%%%%%%%%%%%%%%%%%%%%%%%%%%%%%%%%%%%%%%%%%%%%%%%%%%%

Using the grid diagram formulation of knot Floer homology, Ozsv\'ath, Szab\'o and Thurston defined an invariant of transverse knots in the tight contact 3-sphere. Shortly afterwards, Lisca, Ozsv\'ath, Stipsicz and Szab\'o defined an invariant of transverse knots in arbitrary contact 3-manifolds using open book decompositions. It has been conjectured that these invariants agree where they are both defined. We prove this fact by defining yet another invariant of transverse knots, showing that this third invariant agrees with the two mentioned above.

%%%%%%%%%%%%%%%%%%%%%%%%%%%%%%%%%%%%%%%%%%%%%%%%%%%%%%%	
\end{abstract}
%%%%%%%%%%%%%%%%%%%%%%%%%%%%%%%%%%%%%%%%%%%%%%%%%%%%%%%

%%%%%%%%%%%%%%%%%%%%%%%%%%%%%%%%%%%%%%%%%%%%%%%%%%%%%%%
\section{Introduction} % (fold)
\label{sec:introduction}
%%%%%%%%%%%%%%%%%%%%%%%%%%%%%%%%%%%%%%%%%%%%%%%%%%%%%%%

Since the beginning of modern contact geometry in the work of Eliashberg, Legendrian and transverse knots have played a prominent role in understanding contact structures on 3-manifolds. Transverse knots arise naturally, for example, as binding components of open book decompositions, and both Legendrian and transverse knots can be used to discern subtle geometric information about the contact manifolds they inhabit.

Legendrian knots come equipped with two ``classical'' invariants: the Thurston-Bennequin number and the rotation number.  Transverse knots possess a single classical invariant called the self-linking number.  Knot types whose Legendrian or transverse representatives are classified by their classical invariants are referred to as Legendrian or transversely simple. While it is known that some knot types are simple \cite{EF,EH,EH2}, most appear not to be.  As such, developing and understanding non-classical invariants, capable of distinguishing and classifying Legendrian and transverse knots, is a central aim of contact geometry.

The first non-classical invariant of Legendrian knots, dubbed Legendrian contact homology (LCH), appeared as an outgrowth of Eliashberg and Hofer's work on symplectic field theory \cite{EGH}.  Among its other achievements, LCH provided the first examples of Legendrian non-simple knot types \cite{Ch} and resolved the Legendrian mirror problem \cite{Ng}.  Despite such progress, vanishing properties of LCH prevent one from applying this general theory to classification problems and to the study transverse knots.

A little over a decade later, Ozsv\'ath, Szab\'o and Thurston defined powerful invariants $\lambda$ and $\widehat{\lambda}$ of Legendrian links in the standard contact 3-sphere, which take values in the minus and hat versions of knot Floer homology \cite{OST}.  Their invariants are defined via grid diagrams and are thus combinatorial in nature.  Furthermore, $\lambda$ and $\widehat{\lambda}$ remain unchanged under negative Legendrian stabilization, and, therefore, give rise to transverse link invariants $\theta$ and $\widehat{\theta}$ through Legendrian approximation. We refer to $\lambda,\widehat{\lambda},\theta,\widehat{\theta}$ as the GRID invariants.  %Due to their computability, these new invariants have been readily applied to produce many families of transversely non-simple knot types \cite{NOT,Ver,Ba}.  %Additionally, as a byproduct of their combinatorial structure, it is possible to understand their general behavior under operations like connected sum \cite{Ver} and braid stacking \cite{Ba}.

%%%%%%%%%%%%%%%%%%%%%%%%%%%%%%%%%%%%%%%

In a different direction, Lisca, Ozsv\'ath, Stipsicz and Szab\'o used open book decompositions to define invariants $\Lm$ and $\Lh$ of (null-homologous) Legendrian knots in arbitrary contact 3-manifolds \cite{LOSS}.  These alternate invariants also take values in the minus and hat version of knot Floer homology, though they are defined in far greater generality than their combinatorial counterparts. Both $\Lm$ and $\Lh$ remain unchanged under negative Legendrian stabilization, and may therefore be used as above to define transverse knot invariants $\Tm$ and $\Th$ via Legendrian approximation.  Because they are defined using open book decompositions, these invariants are also more clearly tied to the geometry of Legendrian and transverse knot complements. We refer to $\Lm,\Lh,\Tm,\Th$ as the LOSS invariants.

%There are a few well-understood connections between some of the contact-geometric invariants defined within the sphere of Heegaard Floer theory.  For instance, it follows easily from the construction in \cite{LOSS} that if $L \subset (Y,\xi)$ is Legendrian, then, under the natural maps
%\[
%	\HFKm(-Y,L) \to \HFKh(-Y,L) \hspace{0.25in} \text{and} \hspace{0.25in} \HFKm(-Y,L) \to \HFh(-Y),
%\]
%induced by setting the formal variable $U$ equal to zero at the chain complex level, the invariant $\Lm(L)$ is sent to $\Lh(L)$ and $c(Y,\xi)$, respectively.  In a different direction, the third author, in joint with Stipsicz, established a connection between $\Lh(L)$ and the contact invariant in sutured Floer homology associated to the complement of $L$ \cite{SV}.  This has since been clarified by the second author in joint work with Etnyre and Zarev in their study of Heegaard Floer invariants for contact structures on non-compact 3-manifolds.

It has been conjectured for several years that the GRID invariants agree with the LOSS invariants where both are defined -- for Legendrian and transverse knots in the tight contact 3-sphere, $(S^3,\xi_{std})$. Such a result would have a number of important consequences.  Generally speaking, the GRID invariants are well-suited to computation.  Even for relatively complicated knots, determining whether $\widehat{\lambda},\widehat{\theta}$ are zero or nonzero is often a straightforward process.  This computational ease underpins many of the known results regarding families of transversely non-simple knot types \cite{NOT,Ver,Ba,NK}.  On the other hand, the more geometric construction in \cite{LOSS} makes it possible to establish general properties of the LOSS invariants that are hard to prove for the GRID invariants, like Ozsv\'ath and Stipsicz's result on their behaviors under $(+1)$-contact surgeries \cite{OSti}; or Sahamie's result relating $\Lh(L)$ to the contact invariant of $(+1)$-contact surgery on $L$.  Further, it is conjectured that the Legendrian invariants defined in \cite{OST} and \cite{LOSS} are well-behaved with respect to Lagrangian cobordism. If true, this will almost certainly be proven in the LOSS context.  Knowing that ``LOSS=GRID" would allow one to combine the intrinsic advantages of each and to port results from one realm to the other. This ``LOSS=GRID" equivalence is the content of our main theorem, below.

%%%%%%%%%%%%%%%% Main Theorem Trans %%%%%%%%%%%%%%%%%%%%%
\begin{theorem}\label{thm:LOSS_Grid_trans}
Let $K$ be a transverse knot in the standard contact 3-sphere. There exists a isomorphism of bigraded $(\mathbb{Z}/2\mathbb{Z})[U]$-modules,
\[
	\psi: \HFKm(-S^3,K) \to \HFKm(-S^3,K),
\]
which sends $\Tm(K)$ to $\theta(K)$.
\end{theorem}
%%%%%%%%%%%%%%%% End Main Theorem Trans %%%%%%%%%%%%%%%%%%%%%

Since the Legendrian invariants $\Lm$ and $\lambda$ can each be defined from $\Tm$ and $\theta$ via transverse pushoff (roughly, the inverse of Legendrian approximation), the following is an immediate corollary of Theorem~\ref{thm:LOSS_Grid}.

%%%%%%%%%%%%%%%% Main Theorem Leg %%%%%%%%%%%%%%%%%%%%%
\begin{theorem}\label{thm:LOSS_Grid}
Let $L$ be a Legendrian knot in the standard contact 3-sphere. There exists a isomorphism of bigraded $(\mathbb{Z}/2\mathbb{Z})[U]$-modules,
\[
	\psi: \HFKm(-S^3,L) \to \HFKm(-S^3,L),
\]
which sends $\Lm(L)$ to $\lambda(L)$.
\end{theorem}
%%%%%%%%%%%%%%%% End Main Theorem Leg %%%%%%%%%%%%%%%%%%%%%

\begin{remark}
It follows from the proofs of Theorems \ref{thm:LOSS_Grid_trans} and \ref{thm:LOSS_Grid} that analogous correspondences exist for the Legendrian and transverse invariants defined in the hat version of knot Floer homology.
\end{remark}

Since both the GRID and LOSS invariants are defined using Heegaard diagrams, there is a naive approach to showing that the two agree: construct a sequence of Heegaard moves connecting the two relevant diagrams and show that the associated isomorphism on knot Floer homology sends one invariant to the other.  Versions of this approach have been tried without success by many people.  For one thing, the Heegaard diagrams defining the two types of invariants are radically different and there is no obvious canonical sequence of moves connecting the two.  Second, there is no reason to expect, for any such sequence, the task of tracking the image of the Legendrian invariant under the associated isomorphism to be easy or even combinatorial. With this in mind, we employ a completely different strategy to prove Theorem \ref{thm:LOSS_Grid_trans}.

We start by defining a third invariant of transverse links in arbitrary contact 3-manifolds. Our construction uses Pavelescu's result \cite{EPthesis} that any transverse link $K \subset (Y,\xi)$ can be braided with respect to any open book decomposition $(B,\pi)$ for $(Y,\xi)$. Given such a transverse braid, we define classes $t(K) \in \HFKm(-Y,K)$ and $\widehat{t}(K) \in \HFKh(-Y,K)$ which remain unchanged under positive braid stabilization and positive open book stabilization.  As a result, $t$ and $\widehat{t}$ define transverse invariants, by Pavelescu's analogue of the Transverse Markov Theorem \cite{EPthesis}. We refer to $t,\widehat{t}$ as the BRAID invariants.

Our construction differs from those in \cite{OST} and \cite{LOSS} in a couple interesting ways.  For one thing, our BRAID invariants are manifestly transverse invariants -- they are defined in terms of transverse links rather than in terms of Legendrian links via Legendrian approximation. Furthermore, our invariants lend themselves more naturally to understanding the connections between transverse links, braids and mapping class groups, a rich area of exploration even for knots in $(S^3,\xi_{std})$. As a preliminary step in this direction, we define the notion of a \emph{right-veering} braid, following Honda, Kazez and Mati{\'c} \cite{HKM4}, and prove the theorem below.

\begin{theorem}
\label{thm:rightveering}
Suppose $K$ is a transverse braid with respect to some open book for $(Y,\xi)$. If $K$ is not right-veering, then $\widehat t(K) = 0$.
\end{theorem}
Our invariant also appears to be well-suited to studying transverse satellites. This is something we hope to return to in a future paper.

We prove Theorem \ref{thm:LOSS_Grid_trans} by showing that ``LOSS = BRAID" and that ``BRAID = GRID." For the first correspondence, we find, for any transverse knot $K\subset (Y,\xi)$, a single Heegaard diagram in which the same generator of $\CFKm(-Y,K)$ represents both $\Tm(K)$ and $t(K)$. For the second correspondence, we rely on a reformulation of the BRAID and GRID invariants for transverse knots in $(S^3,\xi_{std})$, as described roughly below.

If $K$ is a transverse link in $ (S^3,\xi_{std})$, then a result of Bennequin \cite{benn} states that $K$ can be braided with respect to the open book decomposition $(U,\pi)$ of $(S^3,\xi_{std})$ whose pages are disks (and whose binding $U$ is therefore the unknot).  The unknot $-U\subset -S^3$ induces an Alexander filtration on the knot Floer chain complex of $K\subset -S^3,$
\[
	\emptyset = \mscr{F}^{-U}_m \subset \mscr{F}^{-U}_{m+1} \subset \dots \subset \mscr{F}^{-U}_n=\CFKm(-S^3,K).
\]
Let \[b = \min\{j\,|\,H_*(\mscr{F}^{-U}_j) \neq 0\},\] and let $H_{t}(\mscr{F}^{-U}_b)$ denote the summand of $H_{*}(\mscr{F}^{-U}_b)$ in the top Maslov grading. In Section~\ref{sec:braid_char}, we show that $H_{t}(\mscr{F}^{-U}_b)$ has rank one and that $t(K)$ can be characterized as the image of the generator of $H_{t}(\mscr{F}^{-U}_b)$ under the map
\[
	H_*(\mscr{F}^{-U}_b) \to \HFKm(-S^3,K)
\]
induced by inclusion.\footnote{This is a slight lie to make the exposition cleaner; see Sections \ref{sec:braid_char}, \ref{sec:comb_char} and \ref{sec:braidgrid} for more accurate statements.} In Section~\ref{sec:comb_char}, we show that the GRID invariant $\theta$ admits the same formulation. The fact that the filtered quasi-isomorphism type of this filtration is an invariant of the link $K\cup -U\subset -S^3$ then implies that there is an isomorphism of knot Floer homology which identifies $t(K)$ with $\theta(K)$.

The astute reader will notice that this reformulation of $t$ is very similar to the way in which the contact invariant $c(Y,\xi)$ is defined in \cite{OS4} -- via a filtration on $\widehat{\rm{CF}}(-Y)$ induced by the connected binding of an open book supporting $(Y,\xi)$. 

Moreover, our reformulations of $t$ and $\theta$ reveal yet another interesting link between knot Floer homology and Khovanov homology. In \cite{pla1}, Plamenevskaya defines a transverse invariant in reduced Khovanov homology which associates to a transverse braid $K\subset (S^3,\xi_{std})$ a class $\psi(K)\in\widetilde{{\rm Kh}}(K)$. The braid axis $U$ of $K$ specifies an embedding of $K$ into a solid torus, which one can think of as the product of an annulus with an interval. The reduced Khovanov skein complex of $K\subset A\times I$, as defined in \cite{APS}, is the associated graded object of a filtration on the Khovanov complex for $K$. Roberts proves in \cite{lrob1} that $\psi(K)$ is characterized with respect to this filtration in the same way that $t$ is with respect to $\mscr{F}^{-U}$. His work partially inspired this aspect of our story.

\subsection*{Organization}
Section \ref{sec:preliminaries} provides background on knot Floer homology, the relationship between transverse and Legendrian links, the constructions of the GRID and LOSS invariants and the correspondence between transverse knots in contact 3-manifolds and braids with respect to open books. We define our BRAID invariant $t$ in Section \ref{sec:new_invt}. In Section \ref{sec:rv}, we prove Theorem \ref{thm:rightveering}. In Section \ref{sec:lossbraid}, we prove the equivalence of $t$ with the LOSS invariant $\Tm$. %We conclude Section \ref{sec:new_invt} with a discussion on right-veering braids, and prove Proposition \ref{prop:rightveering}. 
In Section \ref{sec:braid_char}, we show that $t$ can be reformulated in the manner described above, and, in Section \ref{sec:comb_char}, we prove that the GRID invariant $\theta$ admits the same reformulation. Finally, in Section \ref{sec:braidgrid}, we prove the equivalence of $t$ with $\theta$, completing the proof of Theorem \ref{thm:LOSS_Grid_trans}.

%%%%%%%%%%%%%%%%%%%%%%%%%%%%%%%%%%%%%%%%%%%%%%%%%%%%%%%
\subsection*{Acknowledgements} % (fold)
\label{sub:acknowledgements}
%%%%%%%%%%%%%%%%%%%%%%%%%%%%%%%%%%%%%%%%%%%%%%%%%%%%%%%
We would like to thank Vincent Colin, John Etnyre and Lenny Ng for helpful conversations.  We would also like to thank the Banff International Research Station for hosting us during the workshop ``Interactions between contact symplectic topology and gauge theory in dimensions 3 and 4."  A significant portion of this work was carried out while the authors were in attendance.

% subsection acknowledgements (end)
%%%%%%%%%%%%%%%%%%%%%%%%%%%%%%%%%%%%%%%%%%%%%%%%%%%%%%%
% section introduction (end)
%%%%%%%%%%%%%%%%%%%%%%%%%%%%%%%%%%%%%%%%%%%%%%%%%%%%%%%

%%%%%%%%%%%%%%%%%%%%%%%%%%%%%%%%%%%%%%%%%%%%%%%%%%%%%%%
\section{Preliminaries} % (fold)
\label{sec:preliminaries}
%%%%%%%%%%%%%%%%%%%%%%%%%%%%%%%%%%%%%%%%%%%%%%%%%%%%%%%

%%%%%%%%%%%%%%%%%%%%%%%%%%%%%%%%%%%%%%%%%%%%%%%%%%%%%%%
\subsection{Knot Floer Homology} % (fold)
\label{sub:hfk}
%%%%%%%%%%%%%%%%%%%%%%%%%%%%%%%%%%%%%%%%%%%%%%%%%%%%%%%
This subsection provides a review of knot Floer homology. Our exposition is tailored to our specific needs, and therefore includes some aspects of the theory which are not usually discussed in the literature. We shall assume that the reader has some familiarity with the subject; for a more introductory treatment, see \cite{OS3,OS5}. We work with coefficients in $\F=\mathbb{Z}/2\mathbb{Z}$ throughout this paper.

A multi-pointed Heegaard diagram for an oriented link $L\subset Y$ is an ordered tuple $\mathcal{H}=(\Sigma,\alphas,\betas,\zs_L,\ws_L\cup \ws_f),$ where 
\begin{itemize}
\item $\Sigma$ is a Riemann surface of genus $g$,
\item $\alphas = \{\alpha_1,\dots,\alpha_{g+m+n-1}\}$ and $\betas = \{\beta_1,\dots,\beta_{g+m+n-1}\}$ are sets of pairwise disjoint, simple closed curves on $\Sigma$ such that $\{\alpha_1,\dots,\alpha_{g+m-1}\}$ and $\{\beta_1,\dots,\beta_{g+m-1}\}$ span $g$-dimensional subspaces of $H_1(\Sigma;\mathbb{Z})$,
\item $\zs_L$, $\ws_L$ are sets of $m$ ``linked" basepoints such that every component of the complements $\Sigma - \{\alpha_1,\dots,\alpha_{g+m-1}\}$ and $\Sigma - \{\beta_1,\dots,\beta_{g+m-1}\}$ contains exactly one point of $\zs_L$ and one of $\ws_L$, and
\item $\ws_f$ is a set of $n$ ``free" basepoints such that every component of the complements $\Sigma - \alphas$ and $\Sigma - \betas$ contains exactly one point of $\ws_L\cup\ws_f$.
\end{itemize}
We further stipulate that $Y$ is specified by the Heegaard diagram $(\Sigma,\alphas,\betas)$ and that $L$ is obtained as follows. Fix $m$ disjoint, oriented, embedded arcs in $\Sigma - \alphas$ from points in $\zs_L$ to those in $\ws_L$ and form $\gamma_1^{\alpha},\dots,\gamma_m^{\alpha}$ by pushing their interiors into the handlebody specified by $\alphas$. Similarly, define pushoffs $\gamma_1^{\beta},\dots,\gamma_m^{\beta}$ of oriented arcs in $\Sigma-\betas$ from points in $\ws_L$ to points in $\zs_L$. $L$ is the union \[L = \gamma_1^{\alpha}\cup\dots\cup\gamma_m^{\alpha}\cup\gamma_1^{\beta}\cup\dots\cup\gamma_m^{\beta}.\] Note that the Heegaard diagrams $(\Sigma,\betas,\alphas,\ws_L,\zs_L\cup \ws_f)$ and $(-\Sigma,\alphas,\betas,\zs_L,\ws_L\cup \ws_f)$ both encode the link $L\subset -Y$. We will make use of this fact later.

For each $w\in\ws_L\cup\ws_f$, let $U_w$ be a formal variable corresponding to $w$. Consider the tori $\mathbb{T}_{\alpha}  = \alpha_1\times\dots\times\alpha_{g+m+n-1}$ and $\mathbb{T}_{\beta}  = \beta_1\times\dots\times\beta_{g+m+n-1}$ in $\Sym^{g+m+n-1}(\Sigma)$. The knot Floer complex $\CFKm(\mathcal{H})$ is the free $\mathbb{F}[\{U_w\}_{w\in\ws_L\cup\ws_f}]-$module generated by the elements of $\mathbb{T}_{\alpha}\cap \mathbb{T}_{\beta}$. For $\mathbf{x}, \mathbf{y}\in \mathbb{T}_{\alpha}\cap\mathbb{T}_{\beta}$, a Whitney disk $\phi \in \pi_2(\mathbf{x},\mathbf{y})$ and a suitable path of almost complex structures on $\Sym^{g+m+n-1}(\Sigma)$, we denote by $\M(\phi)$ the moduli space of pseudo-holomorphic representatives of $\phi$. Its formal dimension is given by the Maslov index $\mu(\phi)$. Let $\widehat{\M}(\phi)$ denote the quotient of this moduli space by the translation action of $\mathbb{R}$. For a point $p\in \Sigma-\alphas-\betas$, we denote by $n_{p}(\phi)$ the algebraic intersection number of $\phi$ with $\{p\}\times\Sym^{g+m+n-2}(\Sigma)$. Similarly, for a finite subset $\mathbf{p}=\{p_1,\dots,p_k\}\subset\Sigma-\alphas-\betas$, we define $n_{\mathbf{p}}(\phi)$ to be the sum $n_{p_1}(\phi)+\dots+n_{p_k}(\phi)$.

If $b_1(Y)=0$, then the Maslov grading is an absolute $\mathbb{Q}-$grading on $\CFKm(\mathcal{H})$, specified up to an overall shift by the fact that \[M(\mathbf{x}) - M(\mathbf{y}) = \mu(\phi)-2n_{ \ws_L\cup\ws_f}(\phi)\] for $\mathbf{x}, \mathbf{y}\in \mathbb{T}_{\alpha}\cap\mathbb{T}_{\beta}$ and any $\phi\in\pi_2(\mathbf{x},\mathbf{y})$, and the fact that multiplication by any $U_w$ lowers $M$ by two.  %In general, there is a well-defined $\mathbb{Z}/2\mathbb{Z}-$grading $\CFKm(\mathcal{H})$ whose associated relative grading is the mod 2 reduction of the formula above. 
Suppose that $L$ is an $l-$component link, $L=L_1\cup\dots\cup L_l$. Then we can write $\zs_L = \zs_{L_1}\cup\dots\cup\zs_{L_l}$ and $\ws_L = \ws_{L_1}\cup\dots\cup\ws_{L_l}$, where the sets $\zs_{L_i}$ and $\ws_{L_i}$ specify the component $L_i$. For each $i$, the Alexander grading associated to the oriented knot $L_i$ is an absolute $\mathbb{Q}$-grading, specified up to an overall shift by the fact that \[A_{L_i}(\mathbf{x}) - A_{L_i}(\mathbf{y}) = n_{\zs_{L_i}}(\phi) -n_{\ws_{L_i}}(\phi),\] and the fact that multiplication by $U_w$ lowers $A_{L_i}$ by one for any $w\in\ws_{L_i}$ and multiplication by any other $U_w$ preserves $A_{L_i}$. We use $A(\x)$ to denote the sum $A_{L_1}(\x)+\dots+A_{L_l}(\x)$.

%\begin{remark} When $L$ is a knot in $S^3$, $(M,A:=A_L)$ is a $\mathbb{Z}\oplus\mathbb{Z}-$valued bigrading on $\CFKm(\mathcal{H})$.
%\end{remark}

The differential \[\partial^-:\CFKm(\mathcal{H})\rightarrow \CFKm(\mathcal{H})\] is defined on generators by
\[\partial^-\mathbf{x} = \sum_{\mathbf{y}\in\mathbb{T}_{\alpha}\cap \mathbb{T}_{\beta}}\sum_{\substack{\phi \in \pi_2(\mathbf{x},\mathbf{y}) \\ \mu(\phi)=1 \\
n_{z}(\phi) = 0 \,\,\,\forall z\in \zs_L}} \#\widehat{\M}(\phi)\cdot \prod_{w\in\ws_L\cup\ws_f}U_w^{n_{w}(\phi)}\cdot \mathbf{y}.\] The minus version of knot Floer homology is defined to be \[\HFKm(Y,L) :=\HFKm(\HD):= H_*(\CFKm(\mathcal{H}),\partial^{-}).\] The formal variables $U_w$ and $U_{w'}$ act identically on $\HFKm(Y,L)$ when $w$ and $w'$ correspond to the same component of $L$ (i.e. are in the same $\ws_{L_i}$ for some $i$). We let $U_i$ denote the action of $U_w$ for $w\in\ws_{L_i}$. Then $\HFKm(Y,L)$ is an invariant of the link $L\subset Y$, well-defined up to graded $\mathbb{F}[U_1,\dots,U_l,\{U_w\}_{w\in\ws_f}]-$module isomorphism.

Recall that $\CFKh(\mathcal{H})$ is the chain complex obtained by setting $U_w=0$ for exactly one $w$ in each $\ws_{L_i}$. Let $\widehat\partial$ denote the induced differential, and let \[p:\CFKm(Y,L)\rightarrow \CFKh(Y,L)\] be the natural quotient map. It follows from the discussion above that the hat version of knot Floer homology, \[\HFKh(Y,L):=\HFKh(\HD):=H_*(\CFKh(\mathcal{H}),\widehat\partial),\] is independent, up to graded $\F-$module isomorphism, of which $U_w$ we set to zero in defining $\CFKh(\HD)$. A similar statement of independence holds for the induced map \[p_*:\HFKm(\HD)\rightarrow \HFKh(\HD).\]

We now focus on a slightly different chain complex, $\CFK^{-,\ws_f}(\mathcal{H})$, obtained from $\CFKm(\mathcal{H})$ by setting the $U_w=0$ for all $w\in\ws_f$. This complex and filtered versions of it will play important roles in later sections. The homology \[\HFK^{-,n}(Y,L):=\HFK^{-,n}(\HD):=H_*(\CFK^{-,\ws_f}(\mathcal{H}), \partial^-)\] depends, up to graded $\mathbb{F}[U_1,\dots,U_l]-$module isomorphism, only on the link $L\subset Y$ and the number $n$ of free basepoints. Accordingly, we will usually write $\CFK^{-,\ws_f}(\mathcal{H})$ as $\CFK^{-,n}(\mathcal{H})$.

Recall that any two multi-pointed Heegaard diagrams for $L\subset Y$ are related by a sequence of Heegaard moves in the complement of the basepoints. These moves are isotopy, handleslide, index 1/2 (de)stabilization, and two additional moves that we term \emph{linked} and \emph{free} index 0/3 (de)stabilization, following \cite{mosz}. Two such diagrams with the same numbers of free basepoints can be related using just the first four moves. There are maps associated to the first three moves which induce isomorphisms of $\HFK^{-,n}(Y,L)$.\footnote{Henceforth, when we speak of ``isomorphisms," we mean graded isomorphisms over the appropriate polynomial ring.} In the case of isotopy and handleslide, these maps can be described in terms of pseudo-holomorphic triangle counts; the map associated to index 1/2 (de)stabilization is induced by an isomorphism of complexes. Below, we study the effects of the last two moves on $\HFK^{-,n}(Y,L).$ 

Let $D$ be a region of $\Sigma-\betas$ which contains some $z\in \zs_L$ and $w\in \ws_L$. Linked index 0/3 stabilization is the operation of adding one basepoint to $\zs_L$, one to $\ws_L$, one curve to $\alphas$ and one to $\betas$ in the manner shown in Figure \ref{fig:nonfreestab}. We refer to the new $\alphas, \betas$ curves as $\alpha',\beta'$, the new basepoints as $z',w'$ and the new Heegaard diagram as $\HD'$. In addition, we denote by $x'$ and $y'$ the intersection points between $\alpha'$ and $\beta'$, as shown in the figure. Linked index 0/3 destabilization is the inverse of this operation.

Let $\CFK^{-,n}(\HD')_1$ be the quotient complex of $\CFK^{-,n}(\HD')$ generated by elements of the form $\mathbf{x}\cup \{x'\}$, and let $\CFK^{-,n}(\HD')_2$ be the subcomplex generated by the elements $\mathbf{x}\cup\{y'\}$, for $\mathbf{x}\in \mathbb{T}_{\alpha}\cap\mathbb{T}_{\beta}$. Then, $\CFK^{-,n}(\HD')$ is isomorphic to the mapping cone of \[f:\CFK^{-,n}(\HD')_1\rightarrow \CFK^{-,n}(\HD')_2,\] where $f$ is defined by \[f(\mathbf{x}\cup\{x'\}) = (U_{w'}+U_w)\cdot(\mathbf{x}\cup\{y'\}).\] It follows that the map from $\CFK^{-,n}(\mathcal{H})$ to $\CFK^{-,n}(\HD')$ which sends $\mathbf{x}$ to $\mathbf{x}\cup \{y'\}$ induces an isomorphism on homology. This is the isomorphism we associate to linked index 0/3 stabilization. To linked index 0/3 destabilization, we associate the inverse of this isomorphism.

\begin{figure}[!htbp]
\labellist 
\hair 2pt 
\small
\pinlabel $z'$ at 313 61
\pinlabel $w$ at 232 58
\pinlabel $w'$ at 349 61
\pinlabel $z$ at 383 58
\pinlabel $w$ at 7 61
\pinlabel $z$ at 82 60
\pinlabel $y'$ at 312 95
\pinlabel $x'$ at 310 26
\pinlabel $\alpha'$ at 368 85
\pinlabel $\beta'$ at 220 85

\endlabellist 
\begin{center}
\includegraphics[height = 2.6cm]{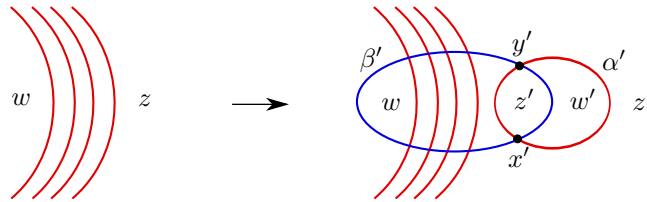}
\caption{\quad Left: before stabilization. Right: after stabilization. }
\label{fig:nonfreestab}
\end{center}
\end{figure}

Free index 0/3 stabilization is the operation of adding one free basepoint to $\ws_f$, one curve to $\alphas$ and one to $\betas$ in the configuration shown in Figure \ref{fig:freestab}. It is important that this new configuration is added in a region of $\Sigma - \alphas - \betas$ containing some element $z\in\zs_L$. As before, we refer to the new $\alphas, \betas$ curves as $\alpha',\beta'$, the new basepoint as $w'$ and the new Heegaard diagram as $\HD'$. Let $x'$ and $y'$ denote the intersection points between $\alpha'$ and $\beta'$ as shown. We say that $\alpha',\beta'$ form a \emph{small configuration} around $w'$. Free index 0/3 destabilization is the inverse of this procedure.

\begin{figure}[!htbp]
\labellist 
\hair 2pt 
\small
\pinlabel $w'$ at 147 61
\pinlabel $z$ at 250 59
\pinlabel $z$ at -30 59
\pinlabel $x'$ at 146 -5
\pinlabel $y'$ at 148 126

\pinlabel $\alpha'$ at 60 105
\pinlabel $\beta'$ at 230 105

\endlabellist 
\begin{center}
\includegraphics[height = 2.1cm]{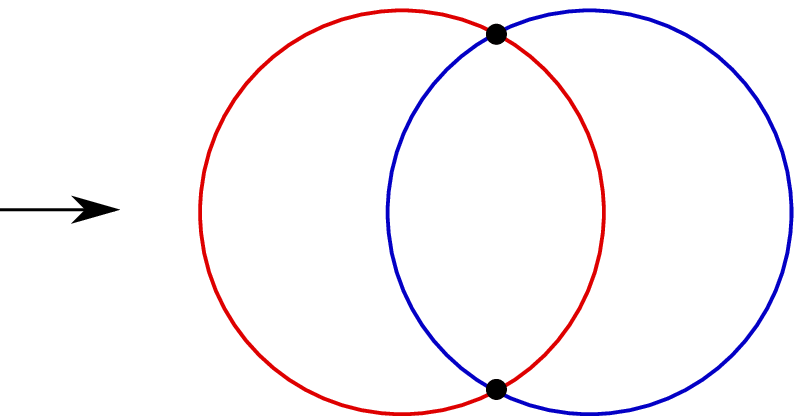}
\caption{\quad Left: before stabilization. Right: after stabilization. }
\label{fig:freestab}
\end{center}
\end{figure}

The complex $\CFK^{-,n+1}(\HD')$ splits as a direct sum of complexes, \begin{equation}\label{eqn:splitting}\CFK^{-,n+1}(\HD')=\CFK^{-,n+1}(\HD')_1\,\oplus \,\CFK^{-,n+1}(\HD')_2,\end{equation} where $\CFK^{-,n+1}(\HD')_1$ is generated by elements of the form $\mathbf{x}\cup \{x'\}$ and $\CFK^{-,n+1}(\HD')_2$ is generated by the elements $\mathbf{x}\cup\{y'\}$, for $\mathbf{x}\in \mathbb{T}_{\alpha}\cap\mathbb{T}_{\beta}$. Note that each direct summand is naturally isomorphic to $\CFK^{-,n}(\mathcal{H})$. In particular, the inclusion \[i:\CFK^{-,n}(\mathcal{H})\rightarrow \CFK^{-,n+1}(\HD'),\] which sends a generator $\mathbf{x}$ to $\mathbf{x}\cup \{x'\}$ is an isomorphism from $\CFK^{-,n}(\mathcal{H})$ to the summand $\CFK^{-,n+1}(\HD')_1[1]$, where the $[1]$ indicates that we have shifted the Maslov grading upwards by $1$. The projection \[j:\CFK^{-,n+1}(\HD')\rightarrow \CFK^{-,n}(\mathcal{H}),\] which sends a generator $\mathbf{x}\cup \{x'\}$ to $\mathbf{x}$ and all others to zero is the inverse of this isomorphism when restricted to $\CFK^{-,n+1}(\HD')_1[1]$. Since the splitting in (\ref{eqn:splitting}) is a splitting of complexes, $i$ and $j$ are chain maps and therefore induce injections and surjections, $i_*$ and $j_*$, on homology. Next, we describe the induced splitting on homology via an action associated to the free basepoint $w'$, following \cite[Section 3.3]{BL}.

For any $w\in\ws_f$, we define a map \[\Psi_w:\CFK^{-,n}(\mathcal{H})\rightarrow \CFK^{-,n}(\mathcal{H})\] which counts disks that pass through $w$ exactly once. Precisely, $\Psi_w$ is defined on generators by \[\Psi_w(\mathbf{x}) = \sum_{\mathbf{y}\in\mathbb{T}_{\alpha}\cap \mathbb{T}_{\beta}}\sum_{\substack{\phi \in \pi_2(\mathbf{x},\mathbf{y}) \\ \mu(\phi)=1 \\
n_{w}(\phi) = 1\\
n_{z}(\phi) = 0 \,\,\,\forall z\in\zs_L}} \#\widehat{\M}(\phi)\cdot \prod_{w\in\ws_L}U_w^{n_{w}(\phi)}\cdot \mathbf{y}.\] A degeneration argument shows that $\Psi_w$ is a chain map and, hence, induces a map $\psi_w$ on homology. Similarly, one can show that $\psi^2_w = 0$ and that $\psi_{w_1}\psi_{w_2} = \psi_{w_2}\psi_{w_1}$ for $w_1\neq w_2 \in \ws_f$. Further degeneration arguments involving holomorphic triangles show that $\psi_w$ commutes with the isomorphisms associated to isotopy and handleslide (c.f. \cite[Proposition 3.6]{BL}). More transparently, $\psi_w$ commutes with the maps associated to index 1/2 and linked index 0/3 (de)stabilization as well as the map associated to free index 0/3 (de)stabilization as long as $w$ is not the free basepoint being added (or removed).

Note that for the Heegaard diagram $\HD'$ obtained from $\mathcal{H}$ via free index 0/3 stabilization as above, $\CFK^{-,n+1}(\HD')_1 = \coker\Psi_{w'}$ and $\CFK^{-,n+1}(\HD')_2 = \ker\Psi_{w'}.$ Therefore, the splitting in (\ref{eqn:splitting}) gives rise to the splitting on homology, \begin{equation*}\label{eqn:splitting2}\HFK^{-,n+1}(Y,L)=\coker\psi_{w'}\,\oplus \,\ker\psi_{w'}.\end{equation*} The inclusion $i_*$ induces an isomorphism \[i_*:\HFK^{-,n}(Y,L)\rightarrow \coker \psi_{w'}[1]\] and the projection $j_*$ restricts to an isomorphism \[j_*: \coker \psi_{w'}[1]\rightarrow \HFK^{-,n}(Y,L).\] Moreover, $j_*\circ i_*$ is the identity. 

The above can be generalized as follows. Suppose that $\HD'$ is obtained from $\mathcal{H}$ via $k$ free index 0/3 stabilizations. Let $w'_1,\dots,w'_k$ denote the free basepoints added in these stabilizations. Let \[i^k:\CFK^{-,n}(\mathcal{H})\rightarrow\CFK^{-,n+k}(\HD')\] denote the obvious composition of $k$ of the inclusion maps $i$, and \[j^k:\CFK^{-,n+k}(\HD')\rightarrow\CFK^{-,n}(\mathcal{H})\] be the composition of $k$ of the projection maps $j$. Then $i^k_*$ induces an isomorphism \[i^k_*:\HFK^{-,n}(Y,L)\rightarrow \big(\cap_{i=1}^k \coker\psi_{w'_i}\big)[k].\] Likewise, $j^k_*$ restricts to an isomorphism \[j^k_*: \big(\cap_{i=1}^k \coker\psi_{w'_i}\big)[k]\rightarrow \HFK^{-,n}(Y,L)\] which is the inverse of the isomorphism above. In Section \ref{sec:braidgrid}, we use the relationships between $i^k_*,j^k_*$ and the free basepoint action together with the fact that this action commutes with the maps induced by Heegaard moves to prove Theorem \ref{thm:LOSS_Grid_trans}.

%We end with a discussion on the effect of free index 0/3 (de)stabilization on $\HFKm(Y,L)$. Let $\CFKm(\tHD)_1$ and $\CFKm(\tHD)_2$ denote the complexes generated by elements of the forms $\x\cup\{x'\}$ and $\x\cup \{y'\}$, respectively. Note that $\CFKm(\tHD)$ is the mapping cone of \[f:\CFKm(\tHD)_1\rightarrow\CFKm(\tHD)_2,\] where $f$ is defined by $f(\x \cup\{x'\}) = U_{w'}\cdot (\y \cup \{y'\})$. It follows that the map from $\CFKm(\HD)$ to $\CFKm(\tHD)$ which sends $\x$ to $\x\cup \{y'\}$ is a chain homotopy equivalence, and, therefore, induces an isomorphism on homology.

%%%%%%%%%%%%%%%%%%%%%%%%%%%%%%%%%%%%%%%%%%%%%%%%%%%%%%%

%%%%%%%%%%%%%%%%%%%%%%%%%%%%%%%%%%%%%%%%%%%%%%%%%%%%%%%
\subsection{Legendrian and Transverse Links} % (fold)
\label{sub:leg_trans}
%%%%%%%%%%%%%%%%%%%%%%%%%%%%%%%%%%%%%%%%%%%%%%%%%%%%%%%

In this short subsection, we collect a few basic notions involving Legendrian and transverse links that are used in defining the invariants discussed in later sections. As in the previous subsection, we assume a fair amount of familiarity with the subject. For a basic reference, see \cite{et2}.

Recall that an oriented link in a contact 3-manifold $(Y,\xi)$ is called \emph{Legendrian} if it is everywhere tangent to $\xi$, and \emph{transverse} if it is everywhere transverse to $\xi$ and intersects $\xi$ positively.\footnote{All contact structures in this paper are cooriented.} We say that two Legendrian (resp. transverse) links are \emph{Legendrian} (resp. \emph{transversely}) \emph{isotopic} if they are isotopic through Legendrian (resp. transverse) links.

A Legendrian link $L$ can be perturbed to a canonical (up to transverse isotopy) transverse link $K$ called the \emph{transverse pushoff} of $L$. Legendrian isotopic links give rise to transversely isotopic pushoffs. Conversely, every transverse link $K$ is transversely isotopic to the pushoff of some Legendrian link $L$; we call such an $L$ a \emph{Legendrian approximation} of $K$. 

There is a local operation on Legendrian links called \emph{negative Legendrian stabilization} which preserves transverse pushoffs;  see \cite{et2}. Said more precisely, the transverse pushoff of a Legendrian link is transversely isotopic to the pushoff of its negative stabilization. Conversely, any two Legendrian approximations of the same transverse link become Legendrian isotopic after sufficiently many negative Legendrian stabilizations \cite{EFM,EH}. It follows that the operation of transverse pushoff gives rise to a one-to-one map from the set of Legendrian links up to Legendrian isotopy and negative stabilization to the set of transverse links up to transverse isotopy. Legendrian approximation is the inverse of this map.

Suppose now that $I$ is an invariant of Legendrian links (up to Legendrian isotopy) which remains unchanged under negative Legendrian stabilization. The discussion above implies that $I$ can also be used to define an invariant $I'$ of transverse links: if $K$ is a transverse link and $L$ is a Legendrian approximation of $K$, define 
\[
	I'(K) := I(L).
\]
As we shall see, this precisely how the transverse GRID and LOSS invariants $\theta$ and $\Tm$ are defined from the Legendrian GRID and LOSS invariants $\lambda$ and $\Lm$, respectively.

%There are two ``classical" invariants of Legendrian links called the \emph{Thurston-Bennequin} and \emph{rotation} numbers, and there is one classical invariant of transverse links called the \emph{self-linking} number. For null-homologous links in arbitrary contact 3-manifolds, these invariants are defined with respect to a Seifert surface. In $(S^3,\xi_{std})$, these integer invariants are independent of Seifert surface, and so we denote them by $tb, r, sl$. For details, see \cite{et2}.

% subsection legendrian_and_transverse_knots (end)
%%%%%%%%%%%%%%%%%%%%%%%%%%%%%%%%%%%%%%%%%%%%%%%%%%%%%%%

%%%%%%%%%%%%%%%%%%%%%%%%%%%%%%%%%%%%%%%%%%%%%%%%%%%%%%%
\subsection{The GRID Invariants} % (fold)
\label{sub:legendrian_invariant_in_combinatorial_knot_floer_homology}
%%%%%%%%%%%%%%%%%%%%%%%%%%%%%%%%%%%%%%%%%%%%%%%%%%%%%%%
In this subsection, we describe the invariants of Legendrian and transverse links in $(S^3,\xi_{std})$ constructed by Ozsv{\'a}th, Szab{\'o} and Thurston in \cite{OST} using grid diagrams. 

Throughout, we shall think of $S^3$ as $\mathbb{R}^3\cup\{\infty\}$. Under this identification, $\xi_{std}$ restricts to the unique tight contact structure on $\mathbb{R}^3$, which we also denote by $\xi_{std}$. We consider the latter to be given by \[\xi_{std} = \ker(dz-ydx).\] There is a natural map from the set of Legendrian links in $(\mathbb{R}^3,\xi_{std})$ to the set of Legendrian links in $(S^3,\xi_{std})$. This map induces a one-to-one correspondence between Legendrian isotopy classes of Legendrian links in $(\mathbb{R}^3,\xi_{std})$ and those in $(S^3,\xi_{std})$. The analogous statements hold for transverse links as well. We shall therefore think of a Legendrian or transverse link as living in $(\mathbb{R}^3,\xi_{std})$ or $(S^3,\xi_{std})$ depending on which is most convenient. 

The construction in \cite{OST} starts with the notion of a \emph{grid diagram}. A grid diagram $G$ is a $k\times k$ square grid along with a collection of $k$ $z$'s and $k$ $w$'s contained in these squares such that every row and column contains exactly one $z$ and one $w$ and no square contains both a $z$ and a $w$. One can produce from $G$ a planar diagram for an oriented link $L\subset S^3$ by drawing a vertical segment in every column from the $z$ basepoint to the $w$ basepoint and a horizontal segment in each row from the $w$ basepoint to the $z$ basepoint so that the horizontal segments pass over the vertical segments, as in Figure \ref{fig:griddiagram}. 

\begin{figure}[!htbp]
\labellist 
\hair 2pt 
\small
\pinlabel $z$ at 12 12
\pinlabel $z$ at 35 35
\pinlabel $z$ at 59 59
\pinlabel $z$ at 105 105
\pinlabel $z$ at 151 129
\pinlabel $w$ at 176 151
\pinlabel $z$ at 129 151
\pinlabel $w$ at 129 176
\pinlabel $z$ at 82 176
\pinlabel $w$ at 151 222
\pinlabel $w$ at 105 199
\pinlabel $w$ at 59 82
\pinlabel $w$ at 35 105
\pinlabel $w$ at 12 129
\pinlabel $z$ at 199 199
\pinlabel $z$ at 223 223
\pinlabel $z$ at 246 246
\pinlabel $w$ at 246 12
\pinlabel $w$ at 223 35
\pinlabel $w$ at 199 59
\pinlabel $z$ at 176 82
\pinlabel $w$ at 82 246

\endlabellist 
\begin{center}
\includegraphics[height = 3.4cm]{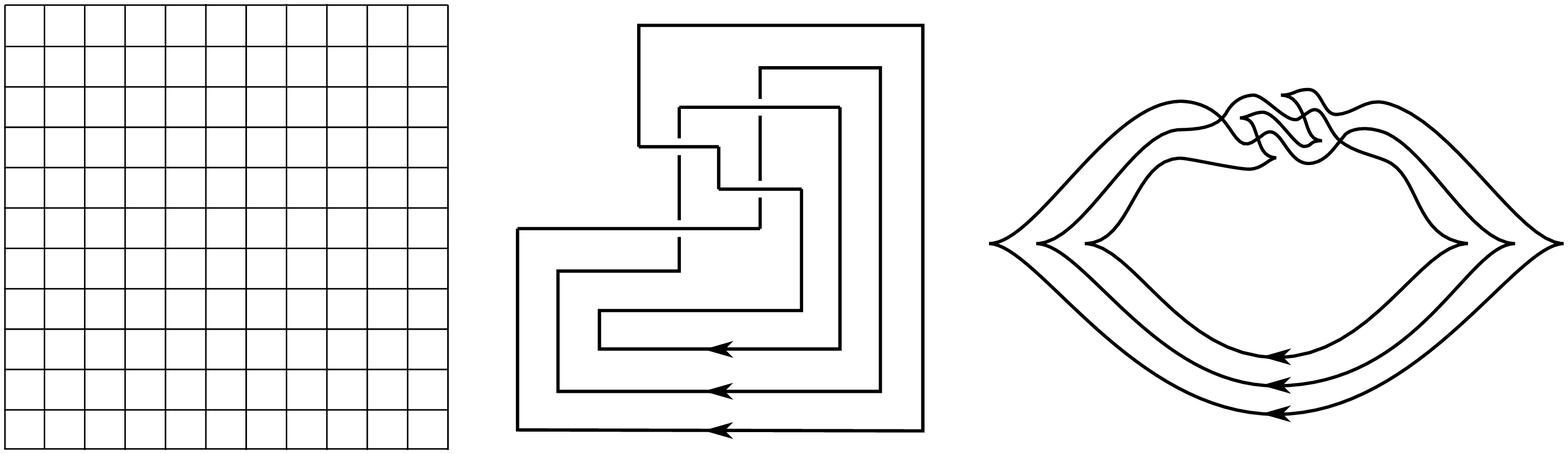}
\caption{\quad Left: the grid diagram $G$. Middle: the link $L$. Right: the Legendrian front projection of $L(G)$. }
\label{fig:griddiagram}
\end{center}
\end{figure}

By rotating this planar diagram $45^{\circ}$ clockwise, and then smoothing the upward and downward pointing corners and turning the leftward and rightward pointing corners into cusps, one obtains the front projection of an oriented Legendrian link in $(\mathbb{R}^3,\xi_{std})$, as in Figure \ref{fig:griddiagram}. Let us denote this Legendrian link by $L(G)$. As discussed above, we may think of $L(G)$ as a Legendrian link in $(S^3,\xi_{std})$. Conversely, every Legendrian link in $(S^3,\xi_{std})$ is Legendrian isotopic to $L(G)$ for some grid diagram $G$.

One associates to $G$ a multi-pointed Heegaard diagram $(T^2,\alphas,\betas,\zs,\ws)$ for $L\subset S^3$, where $T^2$ is the torus obtained by identifying the top and bottom sides of $G$ and the left and right sides of $G$. We orient $T^2$ so that its normal direction points out of the page. The vertical line segments on $G$ correspond to $k$ circles on $T^2$; we denote this set of circles by $\alphas$. Likewise, $\betas$ refers to the set of $k$ circles on $T^2$ corresponding to the horizontal lines on $G$. Finally, $\zs$ and $\ws$ are the sets of $z$ and $w$ basepoints on $T^2$, respectively.

%Suppose $L$ is an oriented Legendrian link in $(S^3, \xi_{std})$. Let $G$ be a grid diagram such that $L$ is Legendrian isotopic to $L(G)$, and let $(T^2,\alphas,\betas,\zs,\ws)$ be the multi-pointed Heegaard diagram associated to $G$. Then $\HD=(-T^2,\alphas,\betas,\zs,\ws)$ is a \emph{nice} multi-pointed Heegaard diagram for $L\subset -S^3$. Note that each generator of $\CFKm(-T^2,\alphas,\betas,\zs,\ws)$ corresponds to a matching of the $\alphas$ and $\betas$ curves; that is, a pairing of each $\alpha_i$ with exactly one $\beta_j$. Furthermore, Manolescu, Ozsv{\'a}th and Sarkar observe that a holomorphic disk of Maslov index one from $\x$ to $\y$ corresponds to an ``empty" rectangle on $-T^2$ whose upper right and lower corners are components of $\x$ and whose upper left and lower right corners are components of $\y$, such that all other components of $\x$ and $\y$ agree \cite{mos}. ``Empty" means that the interior of the rectangle does not contain any components of $\x$ or $\y$.

Suppose $L$ is an oriented Legendrian link in $(S^3, \xi_{std})$. Let $G$ be a grid diagram such that $L$ is Legendrian isotopic to $L(G)$, and let $(T^2,\alphas,\betas,\zs,\ws)$ be the multi-pointed Heegaard diagram associated to $G$. Then $\HD=(-T^2,\alphas,\betas,\zs,\ws)$ is a multi-pointed Heegaard diagram for $L\subset -S^3$. Let $\x$ denote the generator of $\CFKm(-T^2,\alphas,\betas,\zs,\ws)$ consisting of the intersection points at the upper right-hand corners of the squares containing the basepoints in $\zs$. Then $\x$ is a cycle in $\CFKm(-T^2,\alphas,\betas,\zs,\ws)$, and Ozsv{\'a}th, Szab{\'o} and Thurston define the invariant\footnote{In fact, they define two invariants, $\lambda^+$ and $\lambda^-$; our $\Gm$ is their $\lambda^+$. We will not devote any attention to $\lambda^-$ as there is an isomorphism from $\HFKm(-S^3,L)$ to $\HFKm(-S^3,-L)$ which identifies $\lambda^-(L)$ with $\lambda^+(-L)$ \cite[Proposition 1.2]{OST}. } \[\Gm(L):=[\x]\in \HFKm(-S^3,L).\] Likewise, the authors in \cite{OST} define \[\Gh(L):=[\x]\in \HFKh(-S^3,L).\] 

The GRID invariants $\lambda$ and $\Gh$ behave nicely with respect to negative Legendrian stabilization:

\begin{theorem}[Ozsv\'ath-Szab\'o-Thurston \cite{OST}]
\label{thm:stab}
Suppose $L$ is a Legendrian link in $(S^3,\xi_{std})$ and $L_-$ is its negative stabilization. Then there exists an
isomorphism from $\HFKm(-S^3,L)$ to $\HFKm(-S^3,L_-)$ which sends $\lambda(L)$ to $\lambda(L_-)$. The analogous statement holds for $\Gh$.
\end{theorem}

This invariance under negative Legendrian stabilization allows Ozsv{\'a}th, Szab{\'o} and Thurston to define invariants $\theta$ and $\widehat\theta$ of transverse links as suggested in the previous subsection. Namely, if $K$ is a transverse link in $(S^3,\xi_{std})$ and $L$ is a Legendrian approximation of $K$, they define \[\theta(K):=\lambda(L)\in\HFKm(-S^3,K)\] and \[\widehat\theta(K):=\Gh(L)\in\HFKh(-S^3,K).\] Moreover, they show that the Alexander and Maslov gradings of $\theta(K)$ are given by \[A(\theta(K)) = \frac{sl(K)+1}{2}\] and \[ M(\theta(K)) = sl(K)+1.\] 
Here, $sl(K)$ is the \emph{self-linking number} of the transverse link $K$. See \cite{et2} for details.

% subsection legendrian_invariant_in_combinatorial_knot_floer_homology (end)
%%%%%%%%%%%%%%%%%%%%%%%%%%%%%%%%%%%%%%%%%%%%%%%%%%%%%%%

%%%%%%%%%%%%%%%%%%%%%%%%%%%%%%%%%%%%%%%%%%%%%%%%%%%%%%%
\subsection{The LOSS Invariants} % (fold)
\label{sub:loss_invariant}
%%%%%%%%%%%%%%%%%%%%%%%%%%%%%%%%%%%%%%%%%%%%%%%%%%%%%%%

In this subsection, we describe the invariants of Legendrian and transverse knots in arbitrary contact 3-manifolds constructed by Lisca, Ozsv\'ath, Stipsicz and Szab\'o in \cite{LOSS}. %Their construction is inspired by and closely resembles Honda, Kazez and Mati\'c's description of the Ozsv\'ath-Szab\'o contact invariant in \cite{HKM1}. 

Recall that an open book decomposition of a 3-manifold $Y$ is a pair $(B,\pi)$ consisting of an oriented (fibered) link $B\subset Y$ and a locally trivial fibration $\pi:Y-B\rightarrow S^1=\mathbb{R}/\mathbb{Z}$ whose fibers are interiors of Seifert surfaces for the \emph{binding} $B$. The closures of these fibers are called \emph{pages}. It is convenient to record this fibration in the form of an abstract open book $(S,\varphi)$, where $S$ is a compact surface with boundary homeomorphic to a page, and $\varphi$ is a boundary-fixing diffeomorphism of $S$ which encodes the monodromy of the fibration. More precisely, we identify $Y$ with $(S\times [0,1])/\sim_{\varphi},$ where $\sim_{\varphi}$ is the relation defined by 
\begin{align*}
(x,1)\sim_{\varphi}(\varphi(x),0)&, \,\,\,\,\,\, x\in S\\
(x,t)\sim_{\varphi} (x,s)&, \,\,\,\,\,\,x\in\partial S, \,\,\,t,s\in[0,1].
\end{align*} 
Under this identification, $B$ is given by $\partial S \times\{t\}$, and $\pi$ is the map which sends $(x,t)$ to $t$. We denote the page $S\times\{t\}$ by $S_t$. Recall that a \emph{positive stabilization} of $(B,\pi)$ is an open book decomposition corresponding to a fibered link obtained by plumbing a positive Hopf band to $B$. Abstractly, a positive stabilization of $(S,\varphi)$ is an open book $(S',D_{\gamma}\circ\varphi)$, where $S'$ is obtained from $S$ by attaching a 1-handle and $D_{\gamma}$ is the right-handed Dehn twist around a curve $\gamma\subset S'$ which passes through this 1-handle exactly once.

An open book decomposition is said to be \emph{compatible} with a contact structure $(Y,\xi)$ if $\xi = \ker \alpha$ for some 1-form $\alpha\in \Omega^1(Y)$ such that $d\alpha>0$ on the pages of the open book and $\alpha>0$ on its binding. Giroux proved that the map which sends an open book decomposition to a compatible contact structure gives rise to a one-to-one correspondence from set of open book decompositions of $Y$ up to positive stabilization to the set of contact structures on $Y$ up to isotopy \cite{Gi}.

With this background out of the way, we may now define the LOSS invariants. Suppose $L$ is a Legendrian knot in $(Y,\xi)$. Then there exists an open book decomposition $(B,\pi)$ compatible with $(Y,\xi)$ such that $L$ sits as a homologically essential curve on some page $\overline{\pi^{-1}(t)}$. Let $(S,\varphi)$ be an abstract open book corresponding to $(B,\pi)$ in the manner described above. We can think of $L$ as sitting on the page $S_{1/2}$. 

A \emph{basis} for $S$ \emph{adapted} to $L$ is a set of properly embedded arcs $\{a_1,\dots,a_k\}$ in $S$ whose complement is a disk such that $L$ only intersects $a_1$, and does so transversely in a single point.\footnote{Here, we are thinking of $L$ as lying on $S$ via the canonical identification of $S$ with $S_{1/2}$.} Let $\{b_1,\dots,b_k\}$ be another such basis, where each $b_i$ is obtained from $a_i$ by shifting the endpoints of $a_i$ slightly in the direction of the orientation on $\partial S$ and isotoping to ensure that there is a single transverse intersection between $b_i$ and $a_i$, as shown in Figure~\ref{fig:loss}.

\begin{figure}[!htbp]
\labellist 
\hair 2pt 
\small
\pinlabel $S_{1/2}$ at 124 14
\pinlabel $L$ at -7 34

\pinlabel $\alpha_1$ at 98 59
\pinlabel $\beta_1$ at 50 59
\pinlabel $z$ at 73 12
\pinlabel $w$ at 35 41
\pinlabel $z$ at 247 56
\pinlabel $w$ at 208 41

\endlabellist 
\begin{center}
\includegraphics[height = 2cm]{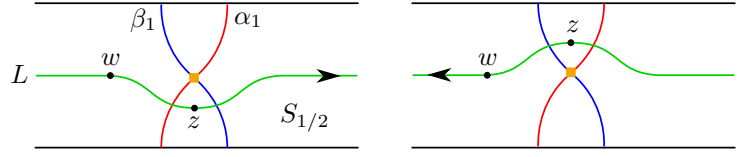}
\caption{\quad Part of the $S_{1/2}$ portion of $\Sigma$ near the intersection of $L$ with $\alpha_1$. Which region contains $z$ depends on the orientation of $L$ as indicated. The orange box represents the component of $\x$ on $\alpha_1\cap\beta_1$.}
\label{fig:loss}
\end{center}
\end{figure}

Let $\Sigma$ denote the closed surface $S_{1/2}\cup-S_0$. For $i=1,\dots,k$, let $\alpha_i=a_i\times\{1/2\}\cup a_i\times\{0\}$ and $\beta_i=b_i\times\{1/2\} \cup \varphi(b_i)\times\{0\}$. Let $w$ be a point on $S_{1/2}$ outside of the thin strips between the $\alpha_i$ and $\beta_i$ curves, and let $z$ be a point on $S_{1/2}$ in one of the two regions between $\alpha_1$ and $\beta_1$, depending on the orientation of $L$, as shown in Figure \ref{fig:loss}. Let $\alphas = \{\alpha_1,\dots,\alpha_k\}$, $\betas = \{\beta_1,\dots,\beta_k\}$, $\zs_L=\{z\}$ and $\ws_L = \{w\}.$ Then $(\Sigma, \alphas,\betas,\zs_L,\ws_L)$ is a doubly-pointed Heegaard diagram for $L\subset Y$, from which it follows that $(\Sigma, \betas,\alphas,\ws_L,\zs_L)$ is a doubly-pointed Heegaard diagram for $L\subset -Y$. Let $\x$ denote the generator of $\CFKm(\Sigma, \betas,\alphas,\ws_L,\zs_L)$ consisting of the intersection points on $S_{1/2}$ between the $\alpha_i$ and $\beta_i$ curves. Then $\x$ is a cycle in $\CFKm(\Sigma, \betas,\alphas,\ws_L,\zs_L)$, and Lisca, Ozsv{\'a}th, Stipsicz and Szab{\'o} define the invariant \[\Lm(L):=[\x]\in\HFKm(-Y,L).\] Likewise, they define \[\Lh(L):=[\x]\in\HFKh(-Y,L).\]

As was the case for the GRID invariants, the LOSS invariants $\Lm$ and $\Lh$ behave nicely with respect to negative Legendrian stabilization:

\begin{theorem}[Lisca-Ozsv\'ath-Stipsicz-Szab\'o \cite{LOSS}]\label{thm:loss_neg}
Suppose $L$ is a Legendrian knot in $(Y,\xi)$ and $L_-$ is its negative stabilization. Then there exists an
isomorphism from $\HFKm(-Y,L)$ to $\HFKm(-Y,L_-)$ which sends $\Lm(L)$ to $\Lm(L_-)$. The analogous statement holds for $\Lh$.
\end{theorem}

Using Theorem \ref{thm:stab}, Lisca, Ozsv\'ath, Stipsicz, Szab\'o define invariants $\Tm$ and $\Th$ of transverse knots as in the previous subsection. Namely, if $K$ is a transverse knot in $(Y,\xi)$ and $L$ is a Legendrian approximation of $K$, they define \[\Tm(K):=\Lm(L)\in\HFKm(-Y,K)\] and \[\Th(K):=\Lh(L)\in\HFKh(-Y,K).\] For a transverse knot $K\subset (S^3,\xi_{std})$, they also prove that the Alexander and Maslov gradings of $\Tm(K)$ are given by \[A(\Tm(K)) = \frac{sl(K)+1}{2}\] and \[ M(\Tm(K)) = sl(K)+1.\] In particular, the gradings of $\Tm$ agree with those of $\Gm$ where both invariants are defined.

%%%%%%%%%%%%%%%%%%%%%%%%%%%%%%%%%%%%%%%%%%%%%%%%%%%%%%%
\subsection{Transverse Braids and Open Books} % (fold)
\label{sub:transverse_braids}
%%%%%%%%%%%%%%%%%%%%%%%%%%%%%%%%%%%%%%%%%%%%%%%%%%%%%%%

In this subsection, we discuss the relationship between transverse knots and braids in open books. We will use this relationship in Section \ref{sec:new_invt} to define our BRAID invariant $t$.

Suppose $(B,\pi)$ is an open book compatible with the contact structure $(Y,\xi)$. A transverse link $K$ in $(Y,\xi)$ is said to be a \emph{braid with respect to} $(B,\pi)$ if $K$ is positively transverse to the pages of $(B,\pi)$. Two such braids are said to be \emph{transversely isotopic with respect to} $(B,\pi)$ if they are transversely isotopic through links which are braided with respect to $(B,\pi)$. The following result of Pavelescu is a generalization of a theorem of Bennequin \cite{benn}.

\begin{theorem}[Pavelescu \cite{EPthesis}]\label{thm:EPput}
Suppose $(B,\pi)$ is an open book compatible with $(Y,\xi)$. Then every transverse link in $(Y,\xi)$ is transversely isotopic to a braid with respect to $(B,\pi)$.
\end{theorem}

If $K$ is a braid with respect to $(B,\pi)$, then $K$ intersects every page of $(B,\pi)$ in the same number of points, say $n$. In this case, we refer to $K$ as an $n$-braid. Let $B_0$ be a binding component of $(B,\pi)$. In \cite{EPthesis}, Pavelescu defines an operation called \emph{positive Markov stabilization around} $B_0$, which is a generalization of the standard positive Markov stabilization for braids in $S^3$. This operation replaces $K$ with an $(n+1)$-braid $K^+$ with respect to $(B,\pi)$ which is transversely isotopic to $K$. See Figure \ref{fig:markovstab} for an illustration. In addition, Pavelescu proves the following generalization of Wrinkle's Transverse Markov Theorem \cite{wrinkle} (see also \cite{osh}).

\begin{theorem}[Pavelescu \cite{EPthesis}]
Suppose $K_1$ and $K_2$ are braids with respect to an open book $(B,\pi)$ compatible with $(Y,\xi)$. Then $K_1$ and $K_2$ are transversely isotopic if and only if there exist positive Markov stabilizations $K_1^+$ and $K_2^+$ around the binding components of $(B,\pi)$ such that $K_1^+$ and $K_2^+$ are transversely isotopic with respect to $(B,\pi)$.
 \end{theorem}
 
\begin{figure}[!htbp]
\labellist 
\hair 2pt 
\small
\pinlabel $B_0$ at 43 90
\pinlabel $B_0$ at 152 90
\pinlabel $K$ at 7 34
\pinlabel $K^+$ at 118 33

\endlabellist 
\begin{center}
\includegraphics[height = 3cm]{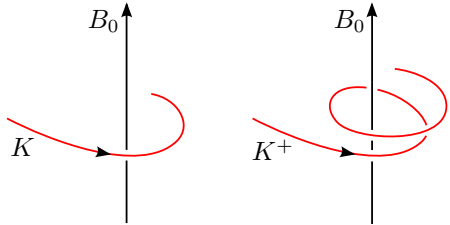}
\caption{\quad On the left, $K$ near the binding component $B_0$. On the right, the positive stabilization $K^+$. }
\label{fig:markovstab}
\end{center}
\end{figure}
 
 The following is an immediate corollary of Pavelescu's work.
 
 \begin{corollary}
 \label{cor:epcor}
Suppose $K_1$ and $K_2$ are braids with respect to open books $(B_1,\pi_2)$ and $(B_2,\pi_2)$ compatible with $(Y,\xi)$. Let $(B,\pi)$ be any common positive stabilization of $(B_1,\pi_1)$ and $(B_2,\pi_2)$, and let $K_1'$ and $K_2'$ denote the induced braids with respect to $(B,\pi)$. Then $K_1$ and $K_2$ are transversely isotopic if and only if there exist positive Markov stabilizations $(K_1')^+$ and $(K_2')^+$ around the binding components of $(B,\pi)$ such that $(K_1')^+$ and $(K_2')^+$ are transversely isotopic with respect to $(B,\pi)$. \qed

 \end{corollary}

Below, we describe how to think about braids in terms of abstract open books. Suppose $(B,\pi)$ is an open book compatible with $(Y,\xi)$ and that $K$ is a transverse link in $(Y,\xi)$ which is an $n$-braid with respect to $(B,\pi)$. Let $(S,\varphi)$ be an abstract open book corresponding to $(B,\pi)$ and let $p_1,\dots,p_n$ be distinct points in the interior of $S$. Then $\varphi$ is isotopic to a diffeomorphism $\widehat\varphi$ of the pair $(S,\{p_1,\dots,p_n\})$ which fixes $\partial S$ point-wise, such that the braid $K$ is corresponds to $(\{p_1,\dots,p_n\}\times[0,1])/\sim_{\widehat\varphi}$ in the identification of $Y$ with $(S\times[0,1])/\sim_{\widehat\varphi}$. We say that the braid $K$ is \emph{encoded} by the \emph{pointed open book} $(S,\{p_1,\dots,p_n\},\widehat\varphi)$. %Let $\Gamma(S,\{p_1,\dots,p_n\})$ be the mapping class group of the pair $(S,\{p_1,\dots,p_n\})$ and $\Gamma(S)$ be the mapping class group of $S$.\footnote{In both mapping class groups, we require that diffeomorphisms and isotopies fix $\partial S$ point-wise.} Then $\widehat\varphi$ is a lift of $\varphi$ with respect to the forgetful map from $\Gamma(S,\{p_1,\dots,p_n\})$ to $\Gamma(S)$.\footnote{We will be sloppy in distinguishing mapping classes from their representative diffeomorphisms.}

%Suppose $K_1$ and $K_2$ are braids with respect to $(B,\pi)$, encoded by $(S,\{p_1,\dots,p_n\},\widehat\varphi_1)$ and $(S,\{p_1',\dots,p_n'\},\widehat\varphi_2)$, where $\widehat\varphi_1$ and $\widehat\varphi_2$. Let $\psi$ be a diffeomorphism of $S$ which is isotopic to the identity and sends $\{p_1,\dots,p_n\}$ to $\{p_1',\dots,p_n'\}$. It follows from the discussion in \cite{EPthesis}, that $K_1$ and $K_2$ are transversely isotopic with respect to $(B,\pi)$ if and only $\widehat\varphi_1$ and $\widehat\varphi_2\circ\psi$ are related by conjugation in $\Gamma(S,\{p_1,\dots,p_n\})$.
% A \emph{$b$-conjugation} is a conjugation with an element from $B_n$, while an \emph{$p$-conjugation} is conjugation with an element from $\Gamma(S)$.
%\begin{theorem}\label{thm:EPconj}\cite{EPthesis}
% Suppose that the transverse braids $T_1$ and $T_2$ with respect to the open book $(B,\pi)$ is described by the lifted monodromies $\widehat{\varphi}_1$ and $\widehat{\varphi}_2$ respectively. Then $T_1$ is isotopic to $T_2$ with respect to $(B,\pi)$ if and only if $\widehat{\varphi}_1$ and $\widehat{\varphi}_2$ are related by a finite sequence of $b$- and $p$-conjugations.
%\end{theorem}

Next, we provide an abstract interpretation of what it means for two braids to be transversely isotopic with respect to a given open book. Suppose $(S,\varphi_1)$ and $(S,\varphi_2)$ are two abstract open books corresponding to $(B,\pi)$ and $K_1$ and $K_2$ are braids with respect to $(B,\pi)$, encoded by the pointed open books $(S,\{p_1^1,\dots,p_n^1\},\widehat\varphi_1)$ and $(S,\{p_1^2,\dots,p_n^2\},\widehat\varphi_2)$. If follows from the discussion in \cite{EPthesis} that $K_1$ and $K_2$ are transversely isotopic with respect to $(B,\pi)$ if and only $\widehat\varphi_2$ is isotopic to $h\circ \widehat\varphi_1\circ h^{-1}$ for some diffeomorphism $h$ which sends $\{p_1^1,\dots,p_n^1\}$ to $\{p_1^2,\dots,p_n^2\}$ (by an isotopy which fixes $\{p_1^2,\dots,p_n^2\}$ and $\partial S$ point-wise).

\begin{figure}[!htbp]
\labellist 
\hair 2pt 
\small
\pinlabel $B_0$ at 91 3
\pinlabel $p_n$ at 23 57
\pinlabel $p_{n+1}$ at 60 8
\pinlabel $\gamma$ at 89 57

\endlabellist 
\begin{center}
\includegraphics[height = 3.4cm]{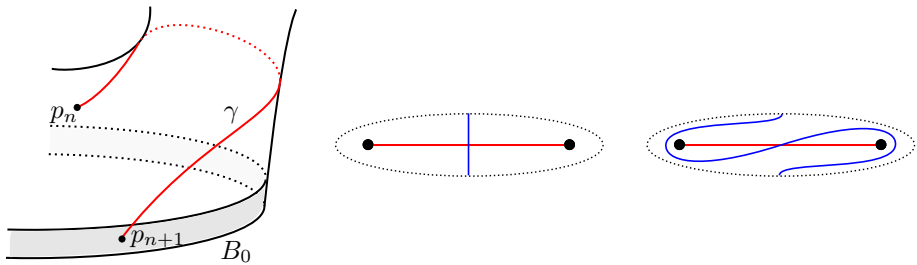}
\caption{\quad On the left, a portion of $S$ near the boundary component $B_0$. The shaded region is $N(B_0)$. The map $d_{\gamma}$ is the identity outside of a neighborhood of the arc $\gamma$. The middle and rightmost figures show such a neighborhood. The diffeomorphism $d_{\gamma}$ is specified on this neighborhood up to isotopy by where it sends the vertical blue arc in the middle diagram. The rightmost diagram shows the image of this arc under $d_{\gamma}$. }
\label{fig:HalfDehn}
\end{center}
\end{figure}

The following is an abstract analogue of positive Markov stabilization. Consider the pointed open book $(S,\{p_1,\dots,p_n\},\widehat\varphi)$. Let $p_{n+1}$ be a point in a collar neighborhood $N(B_0)$ of some boundary component $B_0$ of $S$ such that $\widehat\varphi$ is the identity on $N(B_0)$. Let $\gamma$ be a properly embedded arc in $S-\{p_1,\dots,p_{n-1}\}$ with endpoints on $p_n$ and $p_{n+1}$, and let $d_{\gamma}$ denote the right-handed half twist along $\gamma$, as indicated in Figure \ref{fig:HalfDehn}. Then the pointed open book $(S,\{p_1,\dots,p_{n+1}\},d_{\gamma}\circ\widehat\varphi)$ is said to be a \emph{positive Markov stabilization} of $(S,\{p_1,\dots,p_n\},\widehat\varphi)$. If $K$ is encoded by $(S,\{p_1,\dots,p_n\},\widehat\varphi)$, then any positive Markov stabilization $K^+$ of $K$ is encoded by some positive Markov stabilization $(S,\{p_1,\dots,p_{n+1}\},d_{\gamma}\circ\widehat\varphi)$ of $(S,\{p_1,\dots,p_n\},\widehat\varphi)$, as discussed in \cite{EPthesis}.

It follows from the above discussion and Corollary \ref{cor:epcor} that any two pointed open books which encode transversely isotopic links are related by isotopy, conjugation, positive (open book) stabilization and positive Markov stabilization. We use this fact in the next section to define our invariant $t$.

%%%%%%%%%%%%%%%%%%%%%%%%%%%%%%%%%%%%%%%%%%%%%%%%%%%%%%%
\section{The BRAID Invariants} % (fold)
\label{sec:new_invt}
%%%%%%%%%%%%%%%%%%%%%%%%%%%%%%%%%%%%%%%%%%%%%%%%%%%%%%%

In this section, we define the BRAID invariants $t$ and $\widehat t$. Our construction uses a Heegaard diagram very similar to that used in the construction of the LOSS invariants.

Suppose $K$ is a transverse link in $(Y,\xi)$, braided with respect to some open book decomposition $(B,\pi)$ compatible with $(Y,\xi)$. Let $(S,\varphi)$ be an abstract open book corresponding to $(B,\pi)$, and let $(S,\{p_1,\dots,p_n\},\widehat\varphi)$ be a pointed open book encoding $K$. A \emph{basis} for $(S,\{p_1,\dots,p_n\})$ is a set $\{a_1,\dots,a_k\}$ of properly embedded arcs in $S$ such that $S-\{a_1,\dots,a_k\}$ is a union of $n$ disks each of which contains exactly one point in $\{p_1,\dots,p_n\}$. Let $\{b_1,\dots,b_k\}$ be another such basis, where each $b_i$ is obtained from $a_i$ by shifting the endpoints of $a_i$ slightly in the direction of the orientation on $\partial S$ and isotoping to ensure that there is a single transverse intersection between $b_i$ and $a_i$, as discussed in Subsection \ref{sub:loss_invariant}.

Let $\Sigma$ denote the surface $S_{1/2}\cup -S_0$. For $i=1,\dots,k$, let $\alpha_i=a_i\times\{1/2\}\cup a_i\times\{0\}$ and $\beta_i=b_i\times\{1/2\} \cup \widehat\varphi(b_i)\times\{0\}$, and let $w_i= p_i\times\{1/2\}$ and $z_i=p_i\times\{0\}$. Let $\alphas = \{\alpha_1,\dots,\alpha_k\}$, $\betas = \{\beta_1,\dots,\beta_k\}$, $\zs_K=\{z_1,\dots,z_n\}$ and $\ws_K = \{w_1,\dots,w_n\}.$ Then $(\Sigma, \alphas,\betas,\zs_K,\ws_K)$ is a multi-pointed Heegaard diagram for $K\subset Y$. It follows that $\mathcal{H}=(\Sigma, \betas,\alphas,\ws_K,\zs_K)$ is a multi-pointed Heegaard diagram for $K\subset -Y$. See Figure \ref{fig:texample} for an example. 

Let $\x(\HD)$ denote the generator of $\CFKm(\HD)$ consisting of the intersection points on $S_{1/2}$ between the $\alpha_i$ and $\beta_i$ curves. Note that $\x(\HD)$ is a cycle in $\CFKm(\HD)$. We define \[t(K):=[x(\HD)]\in \HFKm(-Y,K)=\HFKm(\HD),\] and we define $\widehat t(K)\in\HFKh(-Y,K)$ to be the image of $t(K)$ under the natural map $p_*:\HFKm(\HD)\rightarrow \HFKh(\HD)$ discussed in Subsection \ref{sub:hfk}. %It follows that if $t(K)$ is in the image of $U_{z}$ for some $z\in \zs_K$, then $\widehat t(K)=0$. 
The theorem below justifies this notation and implies that $t$ and $\widehat t$ are transverse link invariants.

\begin{figure}[!htbp]
\labellist 
\hair 2pt 
\small
\pinlabel $S$ at 88 190
\pinlabel $S$ at 280 190
\pinlabel $p_1$ at 298 151
\pinlabel $p_2$ at 338 154
\pinlabel $p_1$ at 108 151
\pinlabel $p_2$ at 148 154
\pinlabel $x$ at 25 136
\pinlabel $\gamma$ at 128 116
\pinlabel $S_{1/2}$ at 470 15
\pinlabel $-S_0$ at 470 190
\pinlabel $z_1$ at 489 150
\pinlabel $z_2$ at 529 153
\pinlabel $w_1$ at 489 31
\pinlabel $w_2$ at 529 34

\endlabellist 

\begin{center}
\includegraphics[height = 5cm]{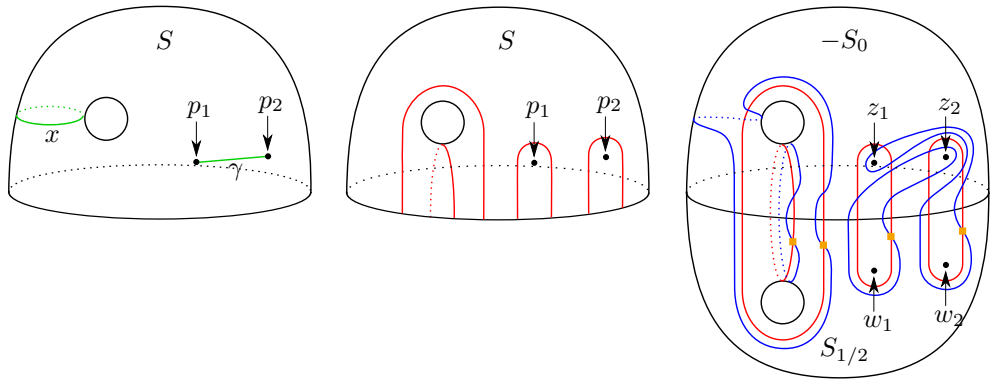}
\caption{\quad An example in which $S$ is the genus one surface with one boundary component and $n=2$. The diagram in the middle shows a basis for $(S,\{p_1,p_2\})$. The diagram on the right is the Heegaard diagram $\HD$ associated to this basis and the pointed open book $(S,\{p_1,p_2\}, D_{x}\circ d_{\gamma}^{-1}),$ where $D_x$ is a right-handed Dehn twist around the curve $x$ shown on the left and $d_{\gamma}^{-1}$ is a left-handed half-twist along the arc $\gamma$. The orange boxes represent the generator $\x(\HD)$. }
\label{fig:texample}
\end{center}
\end{figure}

\begin{theorem}
\label{thm:transvinvt}
Suppose $(S_1,\{p_1^1,\dots,p_n^1\},\widehat\varphi_1)$ and $(S_2,\{p_1^2,\dots,p_n^2\},\widehat\varphi_2)$ are pointed open books encoding braids $K_1$ and $K_2$ with respect to open books $(B_1,\pi_1)$ and $(B_2,\pi_2)$ compatible with $(Y,\xi)$. Let $\HD_1$ and $\HD_2$ be the Heegaard diagrams associated to these pointed open books and bases for $(S_1,\{p_1^1,\dots,p_n^1\})$ and $(S_2,\{p_1^2,\dots,p_n^2\})$. If $K_1$ and $K_2$ are transversely isotopic in $(Y,\xi)$, then there is an isomorphism of graded $\F[U_1,\dots,U_l]-$modules, \[\psi:\HFKm(\HD_1)\rightarrow\HFKm(\HD_2),\] which sends $[\x(\HD_1)]$ to $[\x(\HD_2)]$.\footnote{Here, we are assuming that the smooth link type represented by $K_1$ and $K_2$ has $l$-components and that $U_i$ is the formal variable associated to the $ith$ component.} Likewise, there is an isomorphism of graded $\F-$modules, \[\widehat\psi:\HFKh(\HD_1)\rightarrow\HFKh(\HD_2),\] which sends $[\x(\HD_1)]$ to $[\x(\HD_2)]$.

\end{theorem}

\begin{proof}[Proof of Theorem \ref{thm:transvinvt}]
From the discussion at the end of Subsection \ref{sub:transverse_braids}, it suffices to show that for a pointed open book $(S,\{p_1,\dots,p_n\},\widehat\varphi)$ and basis $\{a_1,\dots,a_k\}$ for $(S,\{p_1,\dots,p_n\})$, each of the five operations 
\begin{enumerate}
\item \label{item:1} change of the basis $\{a_1,\dots,a_k\}$, 
\item \label{item:2} isotopy of $\widehat \varphi$ fixing $\{p_1,\dots,p_n\}$ point-wise,
\item \label{item:3} positive open book stabilization,
\item \label{item:4} conjugation of $\widehat\varphi$,
\item \label{item:5} positive Markov stabilization,
\end{enumerate}
gives rise to an isomorphism on knot Floer homology which sends $t$ to $t$. For (\ref{item:1}), (\ref{item:2}) and (\ref{item:3}), this follows from the proofs of \cite[Proposition 3.4]{HKM1}, \cite[Lemma 3.3]{HKM1} and \cite[Theorem 2.11]{LOSS}, respectively. We first remind the reader of the basic ideas in these proofs before proving invariance of $t$ under (\ref{item:4}) and (\ref{item:5}).

(\ref{item:1}): After relabeling the $a_i$, we can assume $a_1$ and $a_2$ have adjacent endpoints on $\partial S$; that is, there exists an arc $\tau\subset \partial S$ with endpoints on $a_1$ and $a_2$ whose interior is disjoint from all $a_i$. We define $a_1+a_2$ to be the isotopy class (rel. endpoints) of the union $a_1\cup\tau\cup a_2$, as shown on the left in Figure \ref{fig:arcslidemove}. The modification \[\{a_1,a_2,\dots,a_k\}\rightarrow \{a_1+a_2,a_2,\dots,a_k\}\] is called an \emph{arc slide}. Since any two bases for $(S,\{p_1,\dots,p_n\})$ are related by a sequence of arc slides (a trivial extension of \cite[Lemma 3.6]{HKM1}), we need only show that a single arc slide gives rise to an isomorphism on knot Floer homology sending $t$ to $t$. The Heegaard diagrams $\HD$ and $\HD'$ associated to the bases $\{a_1,\dots,a_k\}$ and $\{a_1+a_2,\dots,a_k\}$ above are related by two handleslides: a handleslide of $\beta_1$ over $\beta_2$, followed by a handleslide of $\alpha_1$ over $\alpha_2$. The middle and rightmost portions of Figure \ref{fig:arcslidemove} show parts of the Heegaard triple diagrams associated to these two handleslides. Let $g$ and $f$ denote the corresponding quasi-isomorphisms. There are unique pseudo-holomorphic triangles contributing to $g(\x(\HD))$ and $f(g(\x(\HD)))$ whose domains are unions of small triangles, as shown in Figure \ref{fig:arcslidemove}. From this, it is easy to see that the composition $f_*\circ g_*$ sends $[\x(\HD)]$ to $[\x(\HD')]$.

\begin{figure}[!htbp]
\labellist 
\hair 2pt 
\small
\pinlabel $a_2$ at 155 22
\pinlabel $a_1+a_2$ at 82 22
\pinlabel $a_1$ at 10 22
\pinlabel $S$ at 86 77
\pinlabel $S_{1/2}$ at 390 73
\pinlabel $-S_0$ at 387 107

\pinlabel $S_{1/2}$ at 645 73
\pinlabel $-S_0$ at 642 107

%\pinlabel $w$ at 235 15
%\pinlabel $w$ at 340 15
%\pinlabel $w$ at 430 15

\endlabellist 
\begin{center}
\includegraphics[height = 2.4cm]{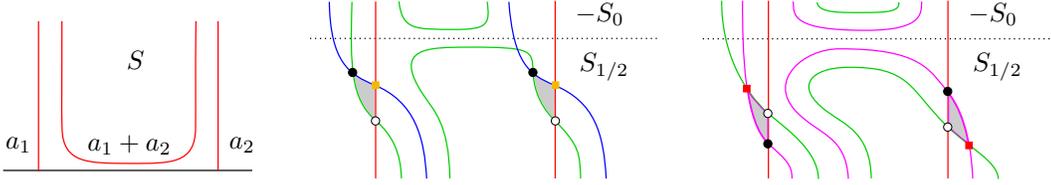}
\caption{\quad On the left, $a_1+a_2$ is the result of arc sliding $a_1$ over $a_2$. This move corresponds to two handleslides whose associated triple diagrams are shown in the middle and right. The small shaded triangles in these diagrams are parts of the two domains contributing to $g(\x(\HD))$ and $f(g(\x(\HD)))$, respectively. The generators $\x(\HD)$ and $\x(\HD')$ are represented by the orange and red boxes in the middle and right figures, respectively.}
\label{fig:arcslidemove}
\end{center}
\end{figure}

(\ref{item:2}): Let $\HD$ and $\HD'$ be the Heegaard associated to the basis $\{a_1,\dots,a_k\}$ and the pointed open books $(S,\{p_1,\dots,p_n\},\widehat \varphi)$ and $(S,\{p_1,\dots,p_n\},\widehat \varphi')$, respectively, where $\widehat\varphi'$ is obtained from $\widehat\varphi$ by an isotopy fixing $\{p_1,\dots,p_n\}$ point-wise. Figure \ref{fig:isotopy} shows a portion of the Heegaard triple diagram associated to this isotopy. Let $f$ denote the corresponding quasi-isomorphism. There is a unique pseudo-holomorphic triangle contributing to $f(\x(\HD))$ whose domain is a union of small triangles of the sort shown in the figure. It follows that $f_*$ sends $[\x(\HD)]$ to $[\x(\HD')].$

\begin{figure}[!htbp]
\labellist 
\hair 2pt 
\small
\pinlabel $S_{1/2}$ at 70 75
\pinlabel $-S_0$ at 67 105
%\pinlabel $w$ at 17 80
%\pinlabel $z$ at 70 80
%\pinlabel $w$ at 155 80
%\pinlabel $z$ at 312 80
\endlabellist 
\begin{center}
\includegraphics[height = 2.8cm]{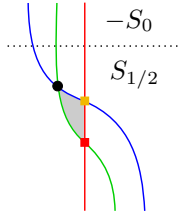}
\caption{\quad Part of the Heegaard triple diagram associated to an isotopy. The generators $\x(\HD)$ and $\x(\HD')$ are represented by the orange and red boxes. The domain contributing to $f(\x(\HD))$ is a union of small shaded triangles of the kind shown in this figure. }
\label{fig:isotopy}
\end{center}
\end{figure}

(\ref{item:3}): Let $(S',\varphi'=D_{\gamma}\circ\varphi)$ be a positive open book stabilization of $(S,\varphi)$. We can choose a basis $\{a_1,\dots,a_k\}$ for $(S,\{p_1,\dots,p_n\})$ which is disjoint from the curve $\gamma$ (a trivial extension of \cite[Section 2.4]{LOSS}). Let $a_{k+1}$ be the co-core of the 1-handle attached in forming $S'$. Let $\HD$ and $\HD'$ be the Heegaard diagrams associated to the bases $\{a_1,\dots,a_k\}$ and $\{a_1,\dots,a_{k+1}\}$ and the pointed open books $(S,\{p_1,\dots,p_n\},\widehat \varphi)$ and $(S',\{p_1,\dots,p_n\},\widehat\varphi')$, respectively. Then $\HD'$ is an index 1/2 stabilization of $\HD$, as indicated in Figure \ref{fig:posstab}. Moreover, the isomorphism between $\CFKm(\HD)$ and $\CFKm(\HD')$ clearly identifies $\x(\HD)$ with $\x(\HD')$.

\begin{figure}[!htbp]
\labellist 
\hair 2pt 
\small
\pinlabel $S_{1/2}$ at 14 130
\pinlabel $-S_{0}$ at 75 130
\pinlabel $S'_{1/2}$ at 145 130
\pinlabel $-S'_{0}$ at 317 130
%\pinlabel $w$ at 17 80
%\pinlabel $z$ at 70 80
%\pinlabel $w$ at 155 80
%\pinlabel $z$ at 312 80
\endlabellist 
\begin{center}
\includegraphics[height = 3.4cm]{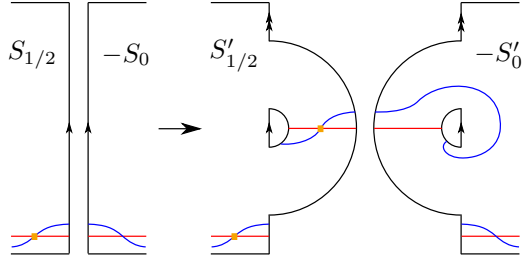}
\caption{\quad On the left, the Heegaard diagram $\HD$. On the right, the Heegaard diagram $\HD'$. This shows that positive open book stabilization corresponds to an index 1/2 stabilization. The orange boxes represent components of $\x(\HD)$ and $\x(\HD')$.}
\label{fig:posstab}
\end{center}
\end{figure}

Below, we show that the moves in (\ref{item:4}) and (\ref{item:5}) also give rise to isomorphisms of knot Floer homology which send $t$ to $t$.

(\ref{item:4}): Let $\HD=(\Sigma,\{\beta_1,\dots,\beta_k\},\{\alpha_1,\dots,\alpha_k\},\ws_K,\zs_K)$ be the Heegaard diagram associated to the basis $\{a_1,\dots,a_k\}$ for $(S,\{p_1,\dots,p_n\})$ and the pointed open book $(S,\{p_1,\dots,p_n\},\widehat\varphi)$. Likewise, let $\HD'=(\Sigma,\{\beta_1',\dots,\beta_k'\},\{\alpha_1',\dots,\alpha_k'\},\ws_K',\zs_K')$ be the Heegaard diagram associated to the basis $\{h(a_1),\dots,h(a_k)\}$ for $(S,\{h(p_1),\dots,h(p_n)\})$ and the pointed open book $(S,\{h(p_1),\dots,h(p_n)\},h\circ\widehat\varphi\circ h^{-1})$. Recall that $\alpha_i=a_i\times\{1/2\}\cup a_i\times\{0\}$ and $\beta_i=b_i\times\{1/2\}\cup \widehat\varphi(b_i)\times\{0\}$. Likewise, $\alpha_i'=h(a_i)\times\{1/2\}\cup h(a_i)\times\{0\}$ and $\beta_i'=h(b_i)\times\{1/2\}\cup h\circ\widehat\varphi\circ h^{-1}(h(b_i))\times\{0\}=h(b_i)\times\{1/2\}\cup h(\widehat\varphi(b_i))\times\{0\}.$ In addition, $\ws_K'=h(\ws_K)$ and $\zs_K'=h(\zs_K)$. The Heegaard diagram $\HD'$ is therefore homeomorphic to $\HD$, so there is a canonical isomorphism of complexes $\CFKm(\HD) \cong \CFKm(\HD')$. Moreover, this isomorphism clearly identifies $\x(\HD)$ with $\x(\HD')$.

(\ref{item:5}): Let $\HD=(\Sigma,\{\beta_1,\dots,\beta_k\},\{\alpha_1,\dots,\alpha_k\},\ws_K,\zs_K)$ be the Heegaard diagram associated to the basis $\{a_1,\dots,a_k\}$ and the pointed open book $(S,\{p_1,\dots,p_n\},\widehat\varphi)$. Let $(S,\{p_1,\dots,p_{n+1}\},d_{\gamma}\circ \widehat\varphi)$ be a positive Markov stabilization of $(S,\{p_1,\dots,p_n\},\widehat\varphi)$, and let $a_{k+1}$ be a boundary parallel arc which splits off a disk containing only the point $p_{n+1}$, so that $\{a_1,\dots,a_{k+1}\}$ is a basis for $(S,\{p_1,\dots,p_{n+1}\})$. Let $w_{i} = p_{i}\times\{1/2\}$ and $z_{i} =p_i\times\{0\}$. Let $\HD'=(\Sigma,\{\beta_1,\dots,\beta_{k+1}\},\{\alpha_1,\dots,\alpha_{k+1}\},\ws_K\cup\{w_{n+1}\},\zs_K\cup \{z_{n+1}\})$ be the Heegaard diagram associated to the basis $\{a_1,\dots,a_{k+1}\}$ and the pointed open book $(S,\{p_1,\dots,p_{n+1}\},d_{\gamma}\circ\widehat\varphi)$. Note that $\beta_{k+1}$ only intersects $\alpha_{k+1}$ and in exactly two points, $x'$ and $y'$, as indicated in Figure \ref{fig:markovstab2}. The diagram $\HD'$ is therefore obtained from $\HD$ by linked index 0/3 stabilization. As explained in Subsection \ref{sub:hfk}, the chain map from $\CFKm(\HD)$ to $\CFKm(\HD')$ which sends a generator $\x$ to $\x\cup \{y'\}$ induces an isomorphism on homology. Now, just observe that $\x(\HD)\cup\{y'\} = \x(\HD')$.

\begin{figure}[!htbp]
\labellist 
\hair 2pt 
\small
\pinlabel $S$ at 140 90
\pinlabel $x'$ at 248 85
\pinlabel $y'$ at 304 45
\pinlabel $S_{1/2}$ at 335 50
\pinlabel $-S_0$ at 335 90
\pinlabel $p_n$ at 20 154
\pinlabel $p_{n+1}$ at 100 138
\pinlabel $z_n$ at 205 154
\pinlabel $z_{n+1}$ at 285 139
\pinlabel $w_{n}$ at 208 -8
\pinlabel $w_{n+1}$ at 289 7
\pinlabel $a_{k+1}$ at 122 110
\pinlabel $\gamma$ at 50 108
\endlabellist 
\begin{center}
\includegraphics[height = 3.4cm]{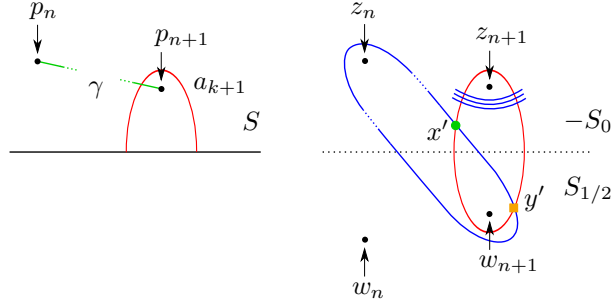}
\caption{\quad On the left, the arcs $a_{k+1}$ and $\gamma$. On the right, the corresponding portion of the Heegaard diagram $\HD'$; $\beta_{k+1}$ is obtained by applying the half-twist $d_{\gamma\times\{0\}}$ to $\alpha_{k+1}$. This shows that positive Markov stabilization corresponds to a linked index 0/3 stabilization. }
\label{fig:markovstab2}
\end{center}
\end{figure}

This completes the proof of Theorem \ref{thm:transvinvt}.
\end{proof}

\section{Right-veering Braids and $t$}
\label{sec:rv}

In this section, we discuss the relationship between our invariant and right-veering braids. Suppose $K$ is a braid with respect to the open book $(B,\pi)$. Let $(S,\varphi)$ be an abstract open book corresponding to $(B,\pi)$ and let $(S,\{p_1,\dots,p_n\},\widehat\varphi)$ be a pointed open book encoding $K$. Suppose $a,b:[0,1]\rightarrow S-\{p_1,\dots,p_n\}$ are properly embedded arcs on $S$ such that $a(0)=b(0)$ and $a(1)=b(1)$. We say that $b$ is \emph{to the right} of $a$ if either $b$ is isotopic to $a$ in $S-\{p_1,\dots,p_n\}$ or, after isotoping $a$ and $b$ in $S-\{p_1,\dots,p_n\}$ so that they intersect efficiently (while keeping their endpoints fixed), the ordered pair $(\dot{b}(0),\dot{a}(0))$ specifies the orientation of $S$.\footnote{Here, we are also using $a$ and $b$ to denote the arcs resulting from these isotopies.} Following \cite{HKM4}, we call the pointed open book $(S,\{p_1,\dots,p_n\},\widehat\varphi)$ \emph{right-veering} if, for every properly embedded arc $a$ in $S-\{p_1,\dots,p_n\}$, $\widehat\varphi(a)$ is to the right of $a$. The lemma below follows easily from the discussion in Subsection \ref{sub:transverse_braids}; one simply observes that conjugation does not affect right-veering-ness.

\begin{lemma}
A pointed open book encoding the braid $K$ is right-veering if and only if \emph{all} pointed open books encoding $K$ are right-veering. \qed
\end{lemma}

We therefore define a braid $K$ with respect to $(B,\pi)$ to be right-veering if all pointed open books encoding $K$ are right-veering. Recall that Theorem \ref{thm:rightveering} claims that if $K$ is not right-veering as a braid with respect to some $(B,\pi)$, then $\widehat t(K)=0.$

\begin{proof}[Proof of Theorem \ref{thm:rightveering}]

Suppose $K$ is an $n$-braid with respect to $(B,\pi)$ which is not right-veering. Let $(S,\{p_1,\dots,p_n\},\widehat\varphi)$ be a pointed open book encoding $K$. Then there exists some properly embedded arc $a$ on $S-\{p_1,\dots,p_n\}$ such that $\widehat\varphi(a)$ is \emph{not} to the right of $a$. Let us first consider the case where $a$ is either non-separating or separates $S$ into two pieces which both contain a point of $\{p_1,\dots,p_n\}$. In this case, we may incorporate $a$ into a basis $\{a=a_1, a_2,\dots,a_k\}$ for $(S,\{p_1,\dots,p_n\})$. 

Let $\mathcal{H}=(\Sigma, \{\beta_1,\dots,\beta_k\},\{\alpha_1,\dots,\alpha_k\},\ws_K,\zs_K)$ be the Heegaard diagram for $K\subset -Y$ associated to this basis and the data $(S,\{p_1,\dots,p_n\},\widehat\varphi)$, as described in the beginning of this section. We may assume that $\alpha_1$ and $\beta_1$ intersect efficiently in $-S_0-\{p_1\times\{0\},\dots,p_n\times\{0\}\}$. The fact that $\widehat\varphi(a)$ is not to the right of $a$ means that an arc of $\beta_1$ forms a half-bigon with an arc of $\alpha_1$ and an arc of the curve $\partial (-S_0)=-\partial (S_{1/2})$, as shown in Figure \ref{fig:rv}. Let $x_1'$ denote the intersection point between $\alpha_1$ and $\beta_1$ on $-S_0$ which is a corner of this half-bigon and, for $i=1,\dots,k$, let $x_i$ denote the unique intersection point between $\alpha_i$ and $\beta_i$ on $S_{1/2}$. 

\begin{figure}[!htbp]
\labellist 
\hair 2pt 
\small
\pinlabel $w$ at 270 17
\pinlabel $w$ at 116 17
\pinlabel $z$ at 228 137
\pinlabel $-S_0$ at 21 101
\pinlabel $S_{1/2}$ at 21 38

\pinlabel $-S_0$ at 175 101
\pinlabel $S_{1/2}$ at 175 38

\pinlabel $x_1$ at 80 40
\pinlabel $x_1'$ at 82 103

\pinlabel $x_1$ at 233 40
\pinlabel $x_1'$ at 235 103
\pinlabel $y$ at 209 168

\endlabellist 
\begin{center}
\includegraphics[height = 3.6cm]{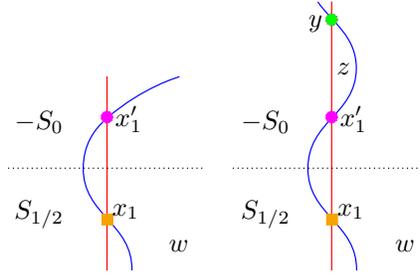}
\caption{\quad A portion of $\HD$ near $\alpha_1$, $\beta_1$ and $\partial (-S_0)=-\partial (S_{1/2})$. Either $\partial^-(\x')=\x$ as on the left, or $\partial^-(\x')=\x+U_z\y$ as on the right. }
\label{fig:rv}
\end{center}
\end{figure}

The obvious bigon in $\Sigma$ with corners at $x_1'$ and $x_1$ is the domain of a class $\phi_0\in \pi_2(\x',\x)$, where $\x'$ is the generator $\{x'_1,x_2,\dots,x_n\}$ and $\x$ is the generator $\{x_1,x_2,\dots,x_n\}$. This class avoids all basepoints, has Maslov index one and supports a unique pseudo-holomorphic representative. Now, suppose $\phi\in\pi_2(\x',\y)$ is class with a pseudo-holomorphic representative for some $\y\neq \x\neq \x'$. Since its domain $D(\phi)$ has no negative multiplicities, $D(\phi)$ cannot have corners at any of $x_2,\dots,x_n$. $D(\phi)$ is therefore a bigon with corners at $x_1'$ and $y$ for some other intersection point $y\in \alpha_1\cap\beta_1\cap -S_0$. In particular, this shows that there can be at most one such $\y$. Since $\alpha_1$ and $\beta_1$ intersect efficiently in $-S_0$, such a bigon exists only if it contains some point $z\in\zs_K$. It follows from this discussion that either no such $\phi$ exists, in which case $\partial^-(\x')=\x$, or that $\partial^-(\x') = \x + U_z\y$. In either case, $t(K)=[\x]$ is in the image of multiplication by $U_z$, from which it follows that $\widehat t(K)=0$.

Next, we consider the case where the arc $a$ separates $S$ into two regions, $R$ and $T$, where $T$ does not contain any of the $p_i$. Let $p_{n+1}$ be a point in a neighborhood $N$ of $\partial S\cap T$ on which $\widehat\varphi$ is the identify, and let $K'$ be the braid with respect to $(B,\pi)$ encoded by the pointed open book $(S,\{p_1,\dots,p_{n+1}\},\widehat\varphi)$. Note that $K'$ is the union of $K$ with an unknot. We can now extend $a$ to a basis for $(S,\{p_1,\dots,p_{n+1}\})$, and it follows from the preceding discussion that $\widehat t(K')=0$. But it is easy to see that $\widehat t(K')=0$ if and only if $\widehat t(K)=0$. 

To prove this, we first pick a basis $\{a_1,\dots,a_k\}$ for $(S,\{p_1,\dots,p_n\})$ and we let $\mathcal{H}=(\Sigma, \{\beta_1,\dots,\beta_k\},\{\alpha_1,\dots,\alpha_k\},\ws_K,\zs_K)$ be the associated Heegaard diagram. Now, let $a_{k+1}$ be a boundary-parallel arc in $S$ which splits off a disk containing just $p_{n+1}$, so that $\{a_1,\dots, a_{k+1}\}$ is a basis for $(S,\{p_1,\dots,p_{n+1}\})$. Let $w_{n+1} = p_{n+1}\times\{1/2\}$ and $z_{n+1} = p_{n+1}\times\{0\}$. Then, \[\mathcal{H}'=(\Sigma, \{\beta_1,\dots,\beta_{k+1}\},\{\alpha_1,\dots,\alpha_{k+1}\},\ws_K\cup \{w_{n+1}\},\zs_K\cup \{z_{n+1}\})\] is the Heegaard diagram for $K'\subset -Y$ associated to this basis and the pointed open book $(S,\{p_1,\dots,p_{n+1}\},\widehat\varphi)$; see Figure \ref{fig:kprime}. Let $x$ be the intersection point $x\in\alpha_{k+1}\cap \beta_{k+1}\cap S_{1/2}$. It is clear that the inclusion \[\iota:\CFKh(\HD)\rightarrow \CFKh(\HD')\] which sends a generator $\x$ to $\x\cup x$ induces an injection $\iota_*$ on homology which sends $\widehat t(K)$ to $\widehat t(K')$. This completes the proof of Theorem \ref{thm:rightveering}.

\begin{figure}[!htbp]
\labellist 
\hair 2pt 
\small
\pinlabel $-S_0$ at 20 79
\pinlabel $S_{1/2}$ at 20 25

\pinlabel $w$ at 133 23
\pinlabel $w_{n+1}$ at 83 34
\pinlabel $z_{n+1}$ at 83 69
\pinlabel $x$ at 88 -5
\pinlabel $x'$ at 90 111

\endlabellist 
\begin{center}
\includegraphics[height = 2.4cm]{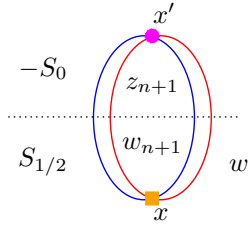}
\caption{\quad A portion of $\HD'$ near $\alpha_{k+1}$, $\beta_{k+1}$ and $\partial (-S_0)=-\partial (S_{1/2})$. Disks cannot leave the intersection point $x$ due to the presence of the basepoints in $\ws_K\cup\{w_{n+1}\}$. The only disks that can enter $x$ are the two canceling bigons from $x'$. It follows that $\iota_*$ is an inclusion.}
\label{fig:kprime}
\end{center}
\end{figure}
\end{proof}

%%%%%%%%%%%%%%%%%%%%%%%%%%%%%%%%%%%%%%%%%%%%%%%%%%%%%%%
\section{LOSS=BRAID} % (fold)
\label{sec:lossbraid}
%%%%%%%%%%%%%%%%%%%%%%%%%%%%%%%%%%%%%%%%%%%%%%%%%%%%%%%
In this short section, we identify the LOSS invariant $\Tm$ with our BRAID invariant $t$. Our main theorem is the following:

\begin{theorem}
\label{thm:LOSS=t}
Let $K$ be a transverse knot in $(S^3,\xi_{std})$. There exists an isomorphism of bigraded $\F[U]-$modules, \[\psi:\HFKm(-S^3,K)\rightarrow\HFKm(-S^3,K),\] which sends $\Tm(K)$ to $t(K)$.
\end{theorem}

We later combine this with Theorem \ref{thm:braid=grid} to prove Theorem \ref{thm:LOSS_Grid_trans}.

\begin{proof}[Proof of Theorem \ref{thm:LOSS=t}]
Suppose $K$ is a transverse knot in $(Y,\xi)$. Then $K$ is transversely isotopic to a binding component $K'$ of some open book decomposition $(B,\pi)$ for $(Y,\xi)$ \cite[Lemma 6.5]{bev2}. Let $(B',\pi')$ be the result of positively stabilizing $(B,\pi)$ in a neighborhood of this binding component $K'$. As shown in \cite[Lemma 3.1]{Ve}, and illustrated on the right side of Figure \ref{fig:stabofpage}, this neighborhood contains both a Legendrian approximation $L$ of $K'$ sitting as a homologically nontrivial curve on a page of $(B',\pi')$, as well as a knot $K''$ which is a 1-braid with respect to $(B',\pi')$ and transversely isotopic to $K'$. 

\begin{figure}[!htbp]
\labellist 
\hair 2pt 
\small
\pinlabel $K'$ at 47 160
\pinlabel $K'$ at 179 169

\pinlabel $L$ at 195 166
\pinlabel $K''$ at 161 166

\endlabellist 
\begin{center}
\includegraphics[height = 6.1cm]{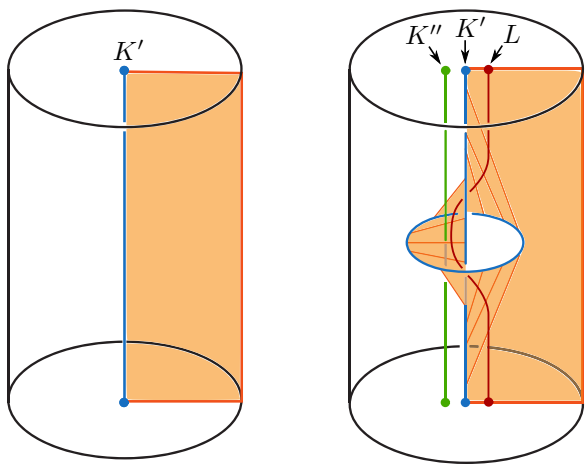}
\caption{\quad On the left, a page of $(B,\pi)$ in a neighborhood of the binding component $K'$. On the right, a page of the stabilized open book $(B',\pi')$ near $K'$. $L$ is a Legendrian approximation of $K'$ and $K''$ intersects each fiber of $\pi'$ in exactly one point.}
\label{fig:stabofpage}
\end{center}
\end{figure}

Since $K$ is transversely isotopic to $K'$, there are isomorphisms of knot Floer homology identifying $t(K)$ with $t(K')$ and $\Tm(K)$ with $\Tm(K')$. It therefore suffices to show that there is an isomorphism of knot Floer homology which identifies $t(K')$ with $\Tm(K')$. But $\Tm(K')$ is defined to be $\Lm(L)$, and $t(K')$ is defined to be $t(K'')$. Thus, to prove Theorem \ref{thm:LOSS=t}, it is enough to show that there is an isomorphism of knot Floer homology which identifies $\Lm(L)$ with $t(K'')$.

\begin{figure}[!htbp]
\labellist 
\hair 2pt 
\small
\pinlabel $K'$ at 47 154
\pinlabel $S$ at 15 82
\pinlabel $K'$ at 133 154
\pinlabel $L$ at 94 152
\pinlabel $S'$ at 101 82
\pinlabel $a_1$ at 286 75
\pinlabel $a_i$ at 232 75

\endlabellist 
\begin{center}
\includegraphics[height = 4.6cm]{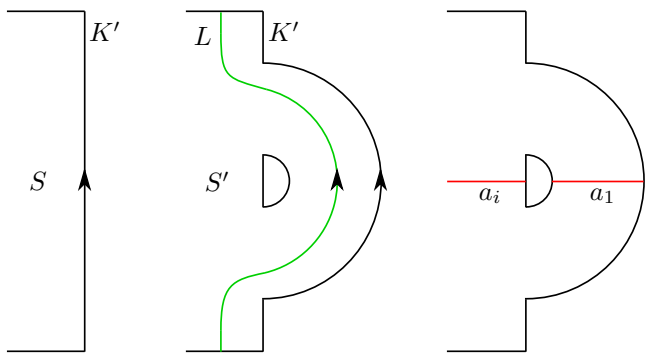}
\caption{\quad On the left, the surface $S$ near the binding component $K'$. In the middle, the stabilized surface $S'$ with the Legendrian approximation $L$ of $K'$. On the right, our choice of basis. We let $a_1$ be the co-core of the attached 1-handle and require that no other $a_i$ intersects the boundary component of $S'$ corresponding to $K'$. }
\label{fig:page}
\end{center}
\end{figure}

\begin{figure}[!htbp]
\labellist 
\hair 2pt 
\small
\pinlabel $S'_{1/2}$ at 18 15
\pinlabel $-S'_0$ at 180 15

\pinlabel $S'_{1/2}$ at 249 15
\pinlabel $-S'_0$ at 412 15

\pinlabel $z$ at 87 87
\pinlabel $w$ at 84 98

\pinlabel $z$ at 360 89
\pinlabel $w$ at 295 88

\endlabellist 
\begin{center}
\includegraphics[height = 4.6cm]{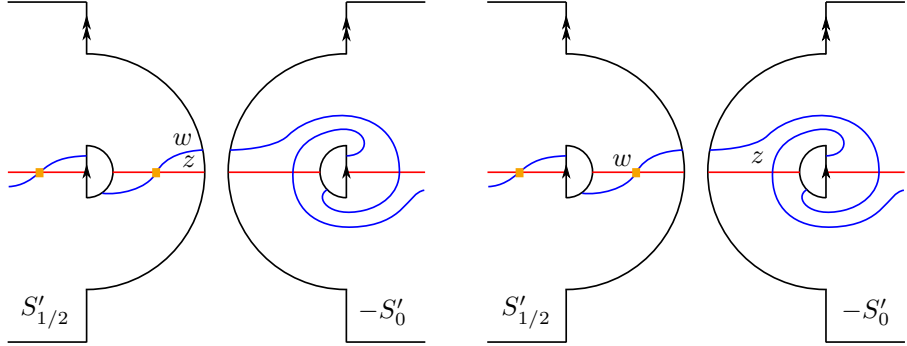}
\caption{\quad On the left, the Heegaard diagram associated to $(S',\varphi')$ and our choice of basis, as used to define the invariant $\Lm(L)$. On the right, the Heegaard diagram associated to the 1-braid $K''$ and the same basis, as used to define the invariant $t(K'')$. The arrows on the boundaries of $S'_{1/2}$ and $-S'_0$ in each case indicate how the surfaces are glued together to form the Heegaard surface. These two Heegaard diagrams only differ in that $z$ and $w$ have been moved slightly. The orange boxes on the left and the right correspond to the generators representing $\Lm(L)$ and $t(K'')$, respectively.}
\label{fig:abstractob}
\end{center}
\end{figure}

Let $(S,\varphi)$ be the abstract open book correponding to $(B,\pi)$, and let $(S',\varphi')$ be the abstract open book corresponding to the positive stabilization $(B',\pi')$. Figure \ref{fig:page} shows a portion of $S'$ near the region in which the stabilization occurred. As illustrated in that figure, we let $a_1$ be the co-core of the 1-handle attached in the stabilization, and we require that no other basis arc intersects the outermost boundary component. The left side of Figure \ref{fig:abstractob} shows the corresponding portion of the Heegaard diagram associated to $(S',\varphi')$ and this choice of basis, as defined in Subsection \ref{sub:loss_invariant}.

Now, observe that if we simply isotop the basepoint $z$ to the $-S'_0$ portion of $\Sigma$, as shown in Figure \ref{fig:abstractob}, then we obtain the Heegaard diagram associated to the transverse 1-braid $K''$ and the above basis, as described in Subsection \ref{sec:new_invt}. In other words, the Heegaard diagrams defining the invariants $\Lm(L)$ and $t(K'')$ are exactly the same (well, isotopic). Moreover, the generators representing $\Lm(L)$ and $t(K'')$ correspond to the same sets of intersection points. This completes the proof of Theorem \ref{thm:LOSS=t}. \end{proof}

%%%%%%%%%%%%%%%%%%%%%%%%%%%%%%%%%%%%%%%%%%%%%%%%%%%%%%%
\section{A Reformulation of the BRAID invariant $t$} % (fold)
\label{sec:braid_char}
%%%%%%%%%%%%%%%%%%%%%%%%%%%%%%%%%%%%%%%%%%%%%%%%%%%%%%%

Suppose that $K$ is a transverse knot in $(S^3,\xi_{std})$, braided with respect to the \emph{standard disk open book decomposition} $(U,\pi)$, where $U$ is an unknot and the fibers $\pi^{-1}(t)$ are open disks. In this section, we give an alternate characterization of $t(K)$ in terms of a filtration on the complex $\CFK^{-,2}(-S^3,K)$ induced by $-U\subset -S^3$. In Section \ref{sec:comb_char}, we show that the invariant $\theta(K)$ admits an identical formulation. We will use these reformulations in Section \ref{sec:braidgrid} to establish an equivalence between the invariants $t$ and $\theta$.

To define this filtration, we first build a special multi-pointed Heegaard diagram for the link $K\cup U\subset S^3$. Let $(S=D^2,\varphi=id)$ be the abstract open book corresponding to $(U,\pi)$. As in Subsection \ref{sub:transverse_braids}, the braid $K$ is specified by the data $(S,\{p_1,\dots,p_n\},\widehat{\varphi})$, for some lift $\widehat{\varphi}$ of $\varphi$. %We use this data to construct a Heegaard diagram adapted to the braid $K$, as in Section \ref{sec:new_invt}. 
Let $a_1,\dots,a_{n-1}$ be basis arcs for $(S,\{p_1,\dots,p_n\})$ such that, for $i=1,\dots,n-1$, the arc $a_i$ separates $S$ into two disks, one of which contains the points $a_1,\dots,a_i$. Recall from Section \ref{sec:new_invt} that this basis and the data $(S,\{p_1,\dots,p_n\},\widehat{\varphi})$ together specify a multi-pointed Heegaard diagram, \[(\Sigma,\{\beta_1,\dots,\beta_{n-1}\},\{\alpha_1,\dots,\alpha_{n-1}\},\ws_K,\zs_K),\] for $K\subset -S^3$, where $\Sigma = S_{1/2}\cup -S_0$; $\alpha_i = a_i\times\{1/2\}\cup a_i\times \{0\}$; $\beta_i = b_i\times\{1/2\}\cup \widehat{\varphi}(b_i)\times \{0\}$ for a small pushoff $b_i$ of $a_i$; and $\zs_K$ and $\ws_K$ consist of the points of the forms $\{p_i\}\times\{1/2\}$ and $\{p_i\}\times\{0\}$, respectively. See Figure \ref{fig:archeeg} for an example in the case that $n=4$.

\begin{figure}[!htbp]
\labellist 
\hair 2pt 
\small
\pinlabel $S_{1/2}$ at 340 35
\pinlabel $-S_0$ at 337 3

\pinlabel $w$ at 248 55
\pinlabel $w$ at 248 78
\pinlabel $w$ at 248 98
\pinlabel $w$ at 248 119

\pinlabel $p_1$ at 72 26
\pinlabel $p_2$ at 72 47
\pinlabel $p_3$ at 72 69
\pinlabel $p_4$ at 72 90

\pinlabel $a_1$ at 91 -5
\pinlabel $a_2$ at 109 1
\pinlabel $a_3$ at 126 12

\pinlabel $U$ at 364 19

\endlabellist 
\begin{center}
\includegraphics[height = 3.7cm]{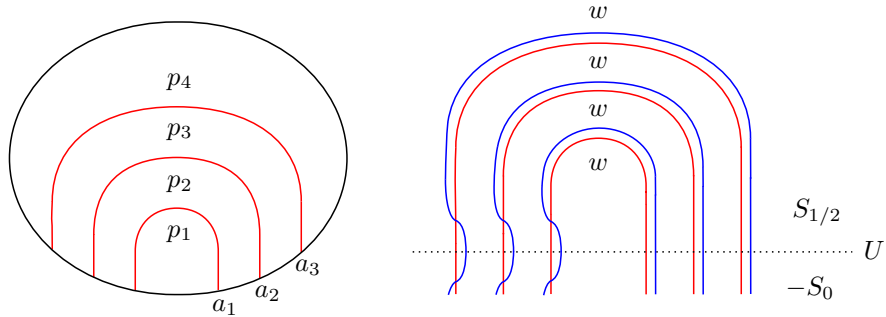}
\caption{\quad On the left, the disk $S$ together with the points $p_i$ and the basis arcs $a_i$ in the case that $n=4$. On the right, a piece of the associated Heegaard diagram near the $S_{1/2}$ portion of $\Sigma$. The $\alpha_i$ are in red, the $\beta_j$ in blue. The set $\ws_K$ consists of the points marked $w$.}
\label{fig:archeeg}
\end{center}
\end{figure}

We may think of $U$ as sitting on $\Sigma$ as the oriented boundary of $S_{1/2}$. Push the curves $\beta_1,\dots,\beta_{n-1}$ along $U$ (in the direction specified by its orientation) via a finger move in such a way that $U$ is the union of two segments, $s_1, s_2$, where $s_1$ only intersects these pushed curves and $s_2$ only intersects the $\alphas$ curves. In an abuse of notation, we denote these pushed curves also by $\beta_1,\dots,\beta_{n-1}$; see Figure \ref{fig:H1}. Let \[\HD_1 = (\Sigma, \{\beta_1,\dots,\beta_{n-1}\},\{\alpha_1,\dots,\alpha_{n-1}\},\ws_K,\zs_K)\] denote the resulting Heegaard diagram. Had we performed the above finger move just inside the $-S_0$ portion of $\Sigma$, we would have obtained an isotopic Heegaard diagram $\HD_1'$, as shown on the right in Figure \ref{fig:H1}. Note that $\HD_1'$ is the diagram associated, in the manner of Section \ref{sec:new_invt}, to the above basis and the data $(S,\{p_1,\dots,p_n\},\widehat{\varphi}')$, where $\widehat{\varphi}'$ is a lift of $\varphi$ which differs from $\widehat{\varphi}$ by an isotopy. The generator $\x_1'$ of $\CFKm(\HD_1')$ whose components are all contained in $S_{1/2}$ therefore represents the transverse invariant $t(K)$. Let $\x_1$ denote the corresponding generator of  $\CFKm(\HD_1)$. It follows that $[\x_1]=t(K)\in \HFKm(\HD_1)$. 

\begin{figure}[!htbp]
\labellist 
\hair 2pt 
\small
\pinlabel $w$ at 78 89
\pinlabel $w$ at 78 110
\pinlabel $w$ at 78 129
\pinlabel $w$ at 78 151

\pinlabel $U$ at -7 56

\pinlabel $w$ at 293 93
\pinlabel $w$ at 293 113
\pinlabel $w$ at 293 134
\pinlabel $w$ at 293 154

\pinlabel $S_{1/2}$ at 198 75
\pinlabel $-S_0$ at 198 36

\endlabellist 
\begin{center}
\includegraphics[height = 4.6cm]{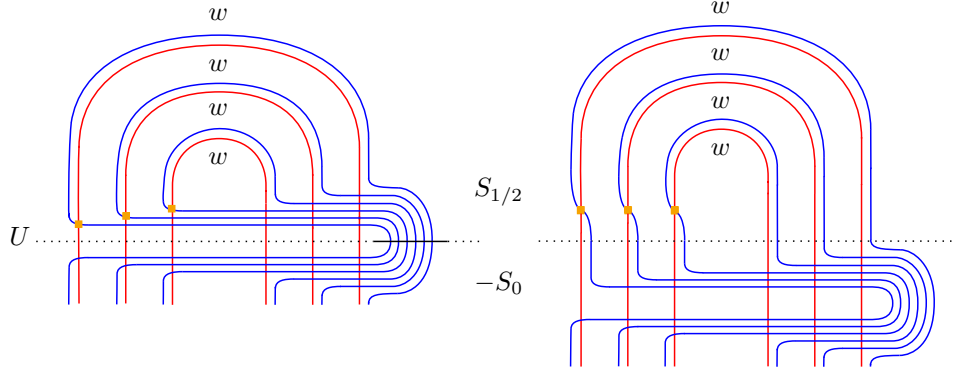}
\caption{\quad On the left, a portion of the diagram $\HD_1$. The solid segment of $U$ is meant to indicate $s_1$. On the right, a portion of the diagram $\HD_1'$. The orange boxes on the left and right are the components of the generators $\x_1$ and $\x_1'$, respectively.}
\label{fig:H1}
\end{center}
\end{figure}

To encode the binding $U$, we place basepoints $z_1,w_1,z_2,w_2$ along $ U'$ in that order so that $z_1,w_1$ are the common endpoints of $s_1,s_2$ and $z_2,w_2$ are contained in $s_2$. Let $\alpha_n, \beta_n,\alpha_{n+1},\beta_{n+1}$ be curves such that $\alpha_n$ encircles $z_1,w_1$; $\beta_n$ encircles $w_1,z_2$; $\alpha_{n+1}$ encircles $z_2,w_2$; and $\beta_{n+1}$ encircles $w_2,z_1$, as shown in Figure \ref{fig:H5}. Let $\alphas=\{\alpha_1,\dots,\alpha_{n+1}\}$, $\betas= \{\beta_1,\dots,\beta_{n+1}\}$, $\zs_U = \{z_1,z_2\}$ and $\ws_U = \{w_1,w_2\}$. Then $(\Sigma,\betas,\alphas, \ws_K\cup\zs_U,\zs_K\cup\ws_U)$ encodes the link $K\cup -U\subset -S^3$ and $\HD_3=(\Sigma,\betas,\alphas, \ws_K,\zs_K\cup\ws_U)$ is a multi-pointed Heegaard diagram for $K\subset -S^3$ with free basepoints $w_1,w_2\in\ws_U$.

\begin{figure}[!htbp]
\labellist 
\hair 2pt 
\small
\pinlabel $z_1$ at 174 49
\pinlabel $w_1$ at 190 49
\pinlabel $z_2$ at 202 49
\pinlabel $w_2$ at 216 49

\pinlabel $w$ at 78 62
\pinlabel $w$ at 78 83
\pinlabel $w$ at 78 103
\pinlabel $w$ at 78 123

\pinlabel $S_{1/2}$ at 195 90
\pinlabel $-S_0$ at 195 5

\endlabellist 
\begin{center}
\includegraphics[height = 4.4cm]{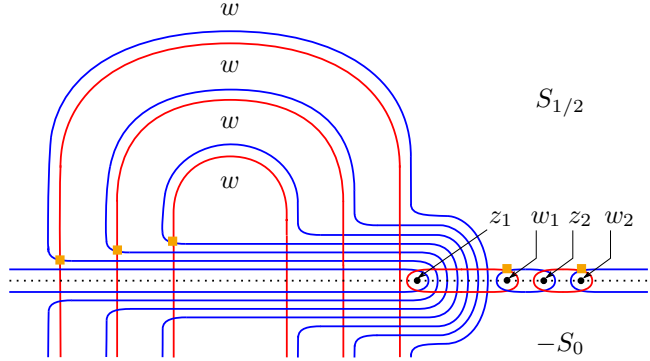}
\caption{\quad A portion of the diagram $\HD_3$. The orange boxes are the components of the generator $\x_3$.}
\label{fig:H3}
\end{center}
\end{figure}

As discussed in Subsection \ref{sub:hfk}, the Alexander grading $A_{-U}$ on $\CFK^{-,2}(\HD_3)$ is specified, up to an overall shift, by the fact that \[A_{-U}(\mathbf{x}) - A_{-U}(\mathbf{y}) = n_{\zs_{U}}(\phi)-n_{\ws_{U}}(\phi)\] for any two generators $\x,\y $ and any $\phi \in \pi_2(\bf{x}, \bf{y})$. This grading induces a filtration \begin{equation}\label{eqn:filt2}\emptyset = \mscr{F}^{-U}_m(\HD_3)\subset\mscr{F}^{-U}_{m+1}(\HD_3)\subset\dots\subset\mscr{F}^{-U}_n(\HD_3)=\rm{CFK}^{-,2}(\HD_3),\end{equation} where $\mscr{F}^{-U}_{k}(\HD_3)$ is generated by $\{\x\in \mathbb{T}_{\alpha}\cup \mathbb{T}_{\beta}\,|\,A_{-U}(\x)\leq k\}.$ The bottommost nontrivial filtration level $\mscr{F}^{-U}_{bot}(\HD_3)$ may be characterized as follows.

\begin{lemma}
\label{lem:afiltration}
A generator of $\CFK^{-,2}(\HD_3)$ is in $\mscr{F}^{-U}_{bot}(\HD_3)$ if and only all of its components are contained in the region $S_{1/2}$.
\end{lemma}

\begin{proof}[Proof of Lemma \ref{lem:afiltration}]
We will prove that the $A_{-U}$ grading of a generator is, up to an overall shift, the number of its components contained in $-S_0$, from which Lemma \ref{lem:afiltration} will follow. First, we show that if $\x$ and $\y$ are generators whose components are all contained in $-S_0$, then $A_{-U}(\x)=A_{-U}(\y)$. 

Suppose $\x$ and $\y$ are as in the previous sentence. Since the $\alphas$ and $\betas$ curves each intersect $-S_0$ in a single arc, we can connect every component of $\x$ to exactly one component of $\y$ along an $\alphas$ arc contained in $-S_0$; likewise, we can connect every component of $\x$ to exactly one component of $\y$ along an $\betas$ arc contained in $-S_0$. The union of these arcs is a collection of closed curves in $-S_0$. Since $-S_0$ is just a disk, this collection of curves bounds a domain contained entirely in $-S_0$. By construction, this domain connects $\x$ and $\y$ and is disjoint from the basepoints in $\zs_U\cup\ws_U$. It follows that $A_{-U}(\x) = A_{-U}(\y)$. 

Let $A_{max}$ denote the $A_{-U}$ grading of generators whose components are all contained in $-S_0$, and suppose we know that for any generator $\y$ with exactly $k$ components in $S_{1/2}$, $A_{-U}(\y) = A_{max}-k$. Now, suppose that $\x$ has exactly $k+1$ components in $S_{1/2}$. Let $x$ be one such component, and suppose that it lies on $\alpha_i$ and $\beta_j$. Let $x'$ be the intersection point on $-S_0$ between the same two curves that is the reflection of $x$ across $U$. Let $\y$ be the generator obtained from $\x$ by replacing $x$ with $x'$. Then $A_{-U}(\y) = A_{max}-k$ by the above hypothesis. There is a bigon from $\y$ to $\x$ which passes once through $z_1$ and avoids the other basepoints in $\zs_U\cup\ws_U$, as shown in Figure \ref{fig:bottomfilt2}. This implies that $A_{-U}(\y)=A_{-U}(\x) + 1$. Hence, $A_{-U}(\x) = A_{max}-(k+1)$. By induction, we have shown that if a generator has exactly $j$ components in $S_{1/2}$, then its $A_{-U}$ grading is $A_{max}-j$ for any $j$. It follows that if a generator has exactly $j$ components in $-S_{0}$, then its $A_{-U}$ grading is $A_{max}-n-1+j$, which is what we set out to prove.
 \end{proof}

\begin{figure}[!htbp]
\labellist 
\hair 2pt 
\small
\pinlabel $S_{1/2}$ at 200 48
\pinlabel $-S_0$ at 200 5
\endlabellist 
\begin{center}
\includegraphics[height = 2.1cm]{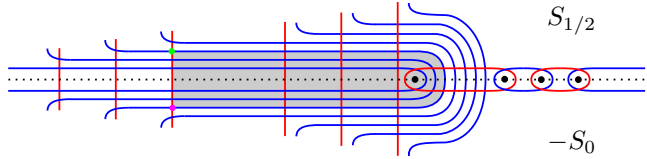}
\caption{\quad A portion of $\HD_3$ near $U$. The components $x'$ and $x$ of $\y$ and $\x$ are shown in pink and green, respectively. The shaded region is the domain of a bigon from $\y$ to $\x$. }
\label{fig:bottomfilt2}
\end{center}
\end{figure}

Let $\x_3$ denote the generator of $\CFK^{-,2}(\HD_3)$ consisting of the components of $\x_1$ together with the intersection points on $\alpha_n\cap\beta_n$ and $\alpha_{n+1}\cap \beta_{n+1}$ with the smallest Maslov grading contributions. Since the components of $\x_3$ are all contained in $S_{1/2}$, Lemma \ref{lem:afiltration} implies that $\x_3$ is in $\mscr{F}^{-U}_{bot}(\HD_3)$.

\begin{lemma}
\label{lem:topmaslov}
The Maslov grading of $\x_3$ is greater than that of any other generator in $\mscr{F}^{-U}_{bot}(\HD_3)$.\end{lemma}

\begin{proof}[Proof of Lemma \ref{lem:topmaslov}]
Suppose $\y$ is a generator in $\mscr{F}^{-U}_{bot}(\HD_3)$. Let $(\y)_j$ and $(\y)^j$ denote the components of $\y$ on $\alpha_j$ and $\beta_j$, respectively. We recursively construct a sequence \[\y=\y_{n+1},\y_n,\dots,\y_{1}=\x_3\] of generators in $\mscr{F}^{-U}_{bot}(\HD_3)$ such that 
\begin{itemize}
\item $(\y_i)_j = (\x_3)_j$ for $i=n,n-1,\dots,1$ and $j=n+1,n,\dots,i$,
\item $M(\y_{i+1})\leq M(\y_i)$, with equality if and only if $\y_{i}=\y_{i+1}$.
\end{itemize}
This will prove Lemma \ref{lem:topmaslov}.

We first construct $\y_n$ from $\y=\y_{n+1}$. If $(\y)_n =(\x_3)_n$, then it must also be that $(\y)_{n+1} = (\x_3)_{n+1}$, and we simply let $\y_n = \y$. Otherwise, there are two cases. For the first, suppose $(\y)_n = (\y)^{n+1}$.  It follows that $(\y)_{n+1}=(\y)^n$. Thus, we let $\y_n$ be the generator obtained from $\y$ by replacing the components $(\y)_n$ and $(\y)_{n+1}$ with $(\x_3)_n$ and $(\x_3)_{n+1}$. Note that there is a disk $\phi\in\pi_2(\y_n,\y)$ whose domain is a square as indicated on the left in Figure \ref{fig:maslovbottom}. This disk has Maslov index one if no other components of $\y$ are contained in its domain, and greater than one otherwise (see the formula for Maslov index in \cite{Li}). Since $\phi$ avoids the basepoints in $\zs_K\cup\ws_U$, $M(\y_n)-M(\y) = \mu(\phi) - 2n_{\zs_K\cup\ws_U}(\phi) \geq 1$. For the second case, suppose $(\y)_n = (\y)^i$ for some $i\neq n,n+1$. Then $(\y)^n=(\y)_j$ for some $j\neq n$. Let $y$ be the intersection point on $\alpha_j\cap \beta_i$ nearest $(\y)^n$ along $\alpha_j$. Let $\y_n$ be the generator obtained from $\y$ by replacing the components $(\y)^n$, $(\y)_n$ and $(\y)_{n+1}$ with $y$, $(\x_3)_n$ and $(\x_3)_{n+1}$. There is a disk $\phi\in\pi_2(\y_n,\y)$ whose domain is a hexagon as indicated on the right in Figure \ref{fig:maslovbottom}. This disk has Maslov index at least one, as above, and avoids the basepoints in $\zs_K\cup\ws_U$, which implies that $M(\y_n)-M(\y) \geq 1$. Observe that $\y_n$ satisfies the desired properties.

\begin{figure}[!htbp]
\labellist 
\hair 2pt 
\small
\pinlabel $w$ at 78 62
\pinlabel $w$ at 78 83
\pinlabel $w$ at 78 102
\pinlabel $w$ at 78 42

\pinlabel $w$ at 319 62
\pinlabel $w$ at 319 83
\pinlabel $w$ at 319 102
\pinlabel $w$ at 319 42

\pinlabel $S_{1/2}$ at 180 70
\pinlabel $S_{1/2}$ at 420 70

\endlabellist 
\begin{center}
\includegraphics[height = 4cm]{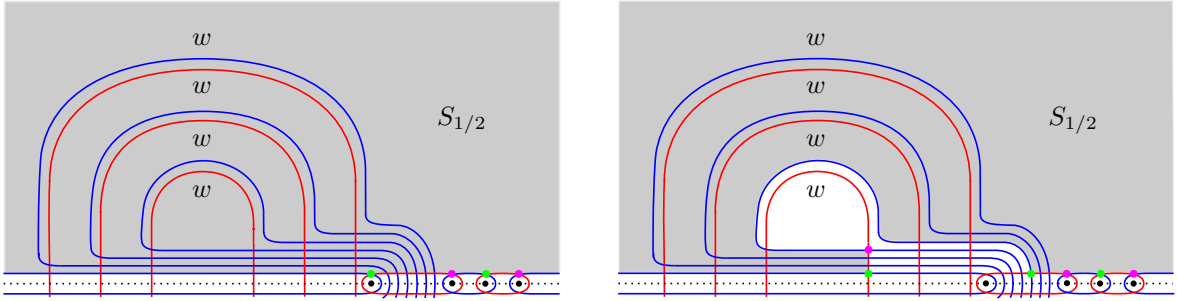}
\caption{\quad On the left, a rectangle connecting $\y_n$ to $\y$. On the right, a hexagon connecting $\y_n$ to $\y$. The pink and green dots reflect the components of $\y_n$ and $\y$, respectively, where these generators differ.}
\label{fig:maslovbottom}
\end{center}
\end{figure}

Suppose we have constructed $\y_{n+1},\y_n,\dots,\y_{i+1}$ with the desired properties. We define $\y_{i}$ from $\y_{i+1}$ in a manner very similar to the way we defined $\y_n$ from $\y$. If $(\y_{i+1})_i = (\x_3)_i$, then we set $\y_i=\y_{i+1}$. Otherwise, there are two cases to consider. For the first, suppose $(\y_{i+1})_i = (\y_{i+1})^i$. Let $\y_i$ be the generator obtained from $\y_{i+1}$ by replacing the component $(\y_{i+1})_i$ with $(\x_3)_i$. There is a disk $\phi\in\pi_2(\y_i,\y_{i+1})$ whose domain is a bigon as indicated on the left in Figure \ref{fig:maslovbottom2}. This disk has Maslov index at least one and avoids the basepoints in $\zs_K\cup\ws_U$, so $M(\y_i)-M(\y_{i+1}) \geq 1$. For the second case, suppose $(\y_{i+1})_i = (\y_{i+1})^k$ for some $k\neq i$. Then $(\y_{i+1})^i=(\y_{i+1})_\ell$ for some $\ell\neq i$. Let $y$ be the intersection point on $\alpha_\ell\cap \beta_k$ nearest $(\y_{i+1})^i$ along $\alpha_\ell$. Let $\y_i$ be the generator obtained from $\y_{i+1}$ by replacing the components $(\y_{i+1})^i$ and $(\y_{i+1})_i$ with $y$ and $(\x_3)_i$. There is a disk $\phi\in\pi_2(\y_{i},\y_{i+1})$ whose domain is a square as indicated on the right in Figure \ref{fig:maslovbottom2}. This disk has Maslov index at least one and avoids the basepoints in $\zs_K\cup\ws_U$, which implies that $M(\y_{i})-M(\y_{i+1}) \geq 1$. Then $\y_i$ has the desired properties. We proceed in this manner to construct the above sequence from $\y$ to $\x_3$.

\begin{figure}[!htbp]
\labellist 
\hair 2pt 
\small
\pinlabel $w$ at 78 62
\pinlabel $w$ at 78 83
\pinlabel $w$ at 78 102
\pinlabel $w$ at 78 42

\pinlabel $w$ at 319 62
\pinlabel $w$ at 319 83
\pinlabel $w$ at 319 102
\pinlabel $w$ at 319 42
\pinlabel $S_{1/2}$ at 180 70
\pinlabel $S_{1/2}$ at 420 70
\endlabellist 
\begin{center}
\includegraphics[height = 3.25cm]{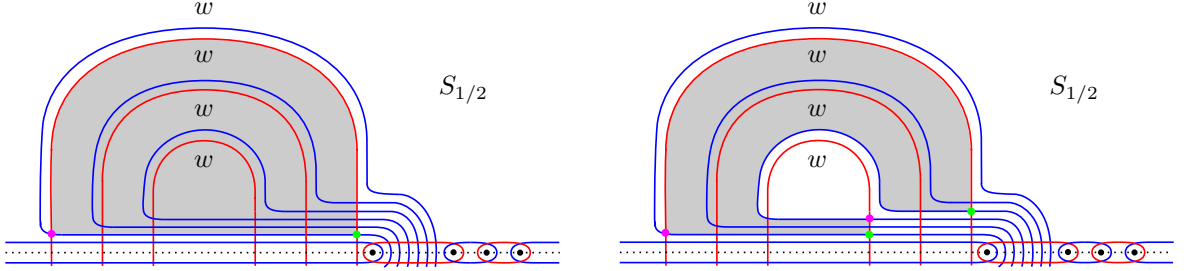}
\caption{\quad On the left, a bigon connecting $\y_i$ to $\y_{i+1}$. On the right, a rectangle connecting $\y_i$ to $\y_{i+1}$. The pink and green dots reflect the components of $\y_i$ and $\y_{i+1}$, respectively, where these generators differ. }
\label{fig:maslovbottom2}
\end{center}
\end{figure}
\end{proof}

In particular, $[\x_3]$ generates the summand $H_{top}(\mscr{F}^{-U}_{bot}(\HD_3))$ of $H_{*}(\mscr{F}^{-U}_{bot}(\HD_3))$ in the top Maslov grading. We may therefore characterize the class $[\x_3]$ more invariantly as follows.

\begin{proposition}
\label{prop:bottomfiltrationa} Let $b = \min\{j\,|\,H_*(\mscr{F}^{-U}_j(\HD_3)) \neq 0\},$ and let $H_t(\mscr{F}^{-U}_{b}(\HD_3))$ denote the summand of $H_*(\mscr{F}^{-U}_{b}(\HD_3))$ in the top Maslov grading. Then $H_t(\mscr{F}^{-U}_{b}(\HD_3))$ has rank one and is generated by $[\x_3].$ \qed

\end{proposition}

\begin{remark}
Though we do not prove it here, one can actually show that \[H_*(\mscr{F}^{-U}_{b}(\HD_3))\cong \F[U_1,\dots,U_l]\otimes_{\F} V\otimes_{\F} V,\] where $U_1,\dots,U_l$ are formal variables corresponding to the $l$ components of the link $K$, and $V=\F_{(0,0)}\oplus\F_{(-1,0)}$, where the subscripts indicate the $(M,A)$ gradings of the summands. 
\end{remark}

Below, we relate the class $[\x_3]$ to $t(K)$ in two steps. 

Let $\HD_2 = (\Sigma,\betas',\alphas', \ws_K,\zs_K\cup\ws_U)$ be the diagram obtained from $\HD_3$ by isotoping $\alpha_n$, $\beta_n$, $\alpha_{n+1}$ and $\beta_{n+1}$ across $z_1$, $z_2$, $z_2$ and $z_1$, respectively, so that the resulting curves $\alpha_n', \beta_n'$ form a small configuration around $w_1$ and $\alpha_{n+1}', \beta_{n+1}'$ form a small configuration around $w_2$, in the sense of Subsection \ref{sub:hfk}. We will assume that $\alpha_i'$ and $\alpha_i$ are related by a small Hamiltonian isotopy for $i\neq n,n+1$, and likewise for $\beta_i'$ and $\beta_i$. Let $\x_2$ be the generator of $\rm{CFK}^{-,2}(\HD_2)$ obtained from $\x_3$ by replacing the components of $\x_3$ on $\alpha_n\cap\beta_n$ and $\alpha_{n+1}\cap \beta_{n+1}$ by the intersection points on $\alpha_n'\cap\beta_n'$ and $\alpha_{n+1}'\cap \beta_{n+1}'$ with the smallest Maslov grading contributions and replacing all other components of $\x_3$ by the corresponding intersection points between the $\alphas'$ and $\betas'$ curves, as shown in Figure \ref{fig:H2}.

\begin{figure}[!htbp]
\labellist 
\hair 2pt 
\small
\pinlabel $w$ at 78 62
\pinlabel $w$ at 78 83
\pinlabel $w$ at 78 102
\pinlabel $w$ at 78 42
\pinlabel $S_{1/2}$ at 180 70

\endlabellist 
\begin{center}
\includegraphics[height = 3.8cm]{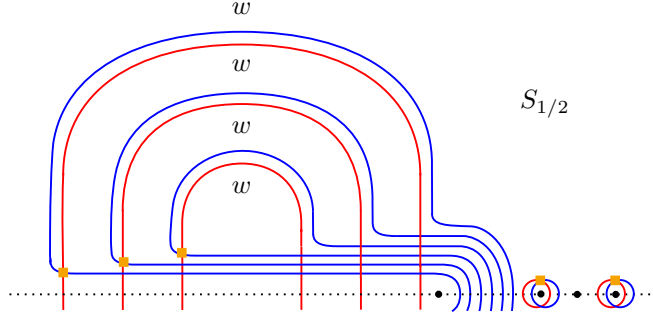}
\caption{\quad A portion of $\HD_2$ near $S_{1/2}$. The orange boxes are the components of $\x_2$.}
\label{fig:H2}
\end{center}
\end{figure}

\begin{proposition}
\label{prop:hslideisota}
The isomorphism \[F_{3,2}:\HFK^{-,2}(\HD_3)\rightarrow\HFK^{-,2}(\HD_2)\] associated to the above sequence of isotopies sends $[\x_3]$ to $[\x_2]$.
\end{proposition}

\begin{proof}[Proof of Proposition \ref{prop:hslideisota}]

Let $\HD_{2.5}=(\Sigma,\betas',\alphas,\ws_K,\zs_K\cup\ws_{U})$.
The map $F_{3,2}$ is induced by the composition $f_{\beta',\alpha,\alpha'}\circ f_{\beta',\beta,\alpha}$, where \begin{align*}
f_{\beta',\beta,\alpha}&: \CFK^{-,2}(\HD_{3})\rightarrow \CFK^{-,2}(\HD_{2.5}),\\
f_{\beta',\alpha,\alpha'}&: \CFK^{-,2}(\HD_{2.5})\rightarrow \CFK^{-,2}(\HD_{2})
\end{align*}
are the pseudo-holomorphic triangle-counting maps associated to the multi-pointed Heegaard triple diagrams $(\Sigma,\betas',\betas,\alphas,\ws_K,\zs_K\cup\ws_{U})$ and $(\Sigma,\betas',\alphas,\alphas',\ws_K,\zs_K\cup\ws_{U})$, respectively. Let $\x_{2.5}$ be the generator of $\CFK^{-,2}(\HD_{2.5})$ obtained from $\x_3$ by replacing the components of $\x_3$ on $\alpha_n\cap\beta_n$ and $\alpha_{n+1}\cap \beta_{n+1}$ by the intersection points on $\alpha_n\cap\beta_n'$ and $\alpha_{n+1}\cap \beta_{n+1}'$ with the smallest Maslov grading contributions, and replacing all other components of $\x_3$ by the corresponding intersection points between the $\alphas$ and $\betas'$ curves. It is not hard to show that $f_{\beta',\beta,\alpha}$ sends $\x_3$ to $\x_{2.5}$ and $f_{\beta',\alpha,\alpha'}$ sends $\x_{2.5}$ to $\x_2$. Below, we illustrate the proof of the first statement; the second follows by a similar argument. These two statements prove Proposition \ref{prop:hslideisot}.

Let $\Theta$ denote the generator of the complex $\CFK^{-,2}(\Sigma,\betas',\betas,\ws_K,\zs_K\cup\ws_{U})$ in its top Maslov grading, let $\y$ be a generator of $\CFK^{-,2}(\HD_{2.5})$ and suppose $u\in \pi_2(\Theta,\x_3,\y)$ is a Whitney triangle (of the sort that would count for $f_{\beta',\beta,\alpha}$) which admits a pseudo-holomorphic representative. We claim that $\y=\x_{2.5}$ and that the domain $D(u)$ must be a disjoint union of the small triangles near the components of $\Theta$, $\x_3$ and $\x_{2.5}$ as shown in Figure \ref{fig:triangles}. In this case, $u$ has a unique pseudo-holomorphic representative, which implies that $f_{\beta',\beta,\alpha}(\x_3) = \x_{2.5}$, proving the proposition. To prove this claim, we analyze the multiplicities of $D(u)$ near these small triangles on $\Sigma$. 

The diagram in the upper right of Figure \ref{fig:triangles} shows said multiplicities near one of these small triangles. Since the region just outside of the triangle but adjacent to the $\betas'$ curve contains a $\ws_K$ basepoint, the multiplicity of $D(u)$ in that region is $0$, as indicated in the figure. The same goes for the region just outside of the triangle and sandwiched between the $\alphas$ and $\betas$ curves. The fact that $\Theta$ and $\x_3$ are corners of $D(u)$ implies that $a=b+c+1$ and $a+d=b+1$. Subtracting the second equation from the first, we have that $-d=c$. The fact that $u$ admits a pseudo-holomorphic representative means that all multiplicities of $D(u)$ are non-negative; therefore, $d=c=0$, which implies that $a=b+1$. If $D(u)$ does not have a corner at the component of $\x_{2.5}$ at the vertex of this small triangle, then $a=-e$ which implies that $a=0$. But this implies that $b=-1$, a contradiction. Therefore, $\y = \x_{2.5}$. But this implies that $a+e=1$, which implies that $a=1$, $e=0$, and $b=0$. In summary, we have shown that the multiplicities of $D(u)$ near this triangles are $a=1$ and $b=c=d=e=0$.
\end{proof}

\begin{figure}[!htbp]
\labellist 
\hair 2pt 
\small
\pinlabel $a$ at 192 73
\pinlabel $b$ at 181 79
\pinlabel $c$ at 203 78
\pinlabel $d$ at 171 60
\pinlabel $e$ at 213 60
\pinlabel $0$ at 192 59
\pinlabel $0$ at 192 96
\pinlabel $w$ at 78 63
\pinlabel $w$ at 78 83
\pinlabel $w$ at 78 103
\pinlabel $w$ at 78 42

\pinlabel $S_{1/2}$ at 193 35
\endlabellist 
\begin{center}
\includegraphics[height = 3.8cm]{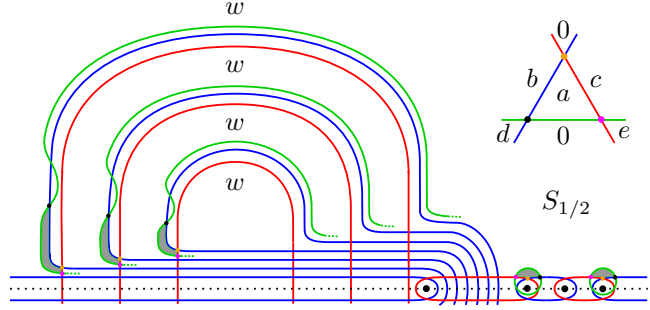}
\caption{\quad A portion of the triple diagram $(\Sigma,\betas',\betas,\alphas,\ws_K,\zs_K\cup\ws_{U})$. The $\betas'$ curves are in green. The generators $\Theta$, $\x_3$ and $\x_{2.5}$ are represented by the black, orange, and pink dots, respectively. The union of the small shaded triangles is the domain of the class $u\in \pi_2(\Theta,\x_3,\x_{2.5})$. }
\label{fig:triangles}
\end{center}
\end{figure}

Observe that $\HD_1$ is the diagram obtained from $\HD_2$ by performing two free index 0/3 destabilizations to remove the basepoints $w_1,w_2$. It is clear that the projection and inclusion maps \begin{align*}j_{2,1}^2&:\CFK^{-,2}(\HD_2)\rightarrow \CFKm(\HD_1),\\
i^2_{1,2}&:\CFKm(\HD_1)\rightarrow \CFK^{-,2}(\HD_2),\end{align*} defined in Subsection \ref{sub:hfk}, send $\x_2$ to $\x_1$ and $\x_1$ to $\x_2$, respectively. Combining this with Proposition \ref{prop:hslideisota}, we have the following.

\begin{proposition}
\label{prop:transversedestaba}
The compositions \begin{align*}(j^2_{2,1})_*\circ F_{3,2}&:\HFK^{-,2}(\HD_3)\rightarrow \HFKm(\HD_1),\\
F_{3,2}^{-1}\circ(i^2_{1,2})_*&:\HFKm(\HD_1)\rightarrow \HFK^{-,2}(\HD_3),\end{align*} send $[\x_3]$ to $t(K)$ and $t(K)$ to $[\x_3]$, respectively.
\end{proposition}

From the identification of $t(K)$ with the LOSS invariant $\Tm(K)$ in the previous section, we know that $M(t(K)) = sl(K)+1$. Recall from Subsection \ref{sub:hfk} that the maps $i^2_{1,2}$ and $j^2_{2,1}$ increase Maslov grading by $2$. It follows that the Maslov grading of $\x_3$ is $M(\x_3) = sl(K)-1$. In particular, $t=top=sl(K)-1$. We will use this last fact in the next section to prove an analogue of Proposition \ref{prop:bottomfiltrationa}.

%%%%%%%%%%%%%%%%%%%%%%%%%%%%%%%%%%%%%%%%%%%%%%%%%%%%%%%
\section{A Reformulation of the GRID Invariant $\theta$} % (fold)
\label{sec:comb_char}
%%%%%%%%%%%%%%%%%%%%%%%%%%%%%%%%%%%%%%%%%%%%%%%%%%%%%%%

In this section, we give an alternate formulation of the GRID invariant $\theta$ in terms of the filtration on the knot Floer complex of a transverse braid that is induced by its braid axis. This reformulation is identical to that described in Section \ref{sec:braid_char} for our BRAID invariant $t$. As mentioned in the previous section, our plan is to combine these reformulations in Section \ref{sec:braidgrid} to establish an equivalence between $t$ and $\theta$.

Suppose $K$ is a transverse link in $(S^3,\xi_{std})$. As in Subsection \ref{sub:legendrian_invariant_in_combinatorial_knot_floer_homology}, we may think of $K$ as a transverse link in $(\mathbb{R}^3,\xi_{std})$. Let $\xi_{rot}$ denote the contact structure on $\mathbb{R}^3$, given in cylindrical coordinates by \[\xi_{rot} = \ker(dz+r^2d\theta).\] The map $\phi:(\mathbb{R}^3,\xi_{rot})\rightarrow (\mathbb{R}^3,\xi_{std})$ defined by \[\phi(x,y,z) = (x, 2y, xy+z)\] is a contactomorphism \cite{NK}. A result of Bennequin states that $\phi^{-1}(K)$ is transversely isotopic in $(\mathbb{R}^3,\xi_{rot})$ to a transverse braid $T_{\beta}$ about the $z$-axis \cite{benn}, as shown in Figure \ref{fig:braidleg}. It follows from that $K$ is transversely isotopic to $\phi(T_{\beta})$, which is the positive transverse pushoff of the Legendrian braid $L_{\beta}$ shown in Figure \ref{fig:braidleg} \cite{NK}. Therefore, $\theta(K) = \lambda(L_{\beta}).$

\begin{figure}[!htbp]
\labellist 
\hair 2pt 
\small
\pinlabel \rotatebox{-3}{$\beta$} at 45 179
\pinlabel $\beta$ at 215 231
\pinlabel $\beta$ at 447 234
\pinlabel $z$ at 85 275
\pinlabel $z$ at 265 274
\pinlabel $z$ at 456 272
\pinlabel $x$ at 350 201
\pinlabel $x$ at 539 196

\endlabellist 
\begin{center}
\includegraphics[height = 5.3cm]{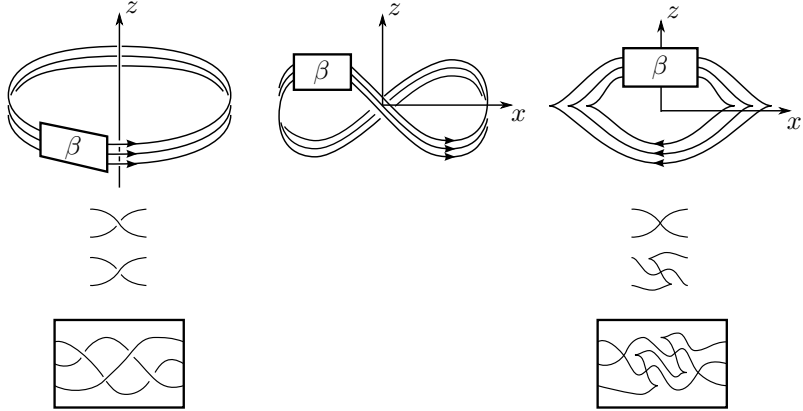}
\caption{\quad The top row shows $T_{\beta}$, $\phi(T_{\beta})$ and $L_{\beta}$ from left to right (we have drawn $T_{\beta}$ as a 3-strand braid). The next row shows how positive and negative crossings in $T_{\beta}$ are converted to crossings in $L_{\beta}$. The bottom row shows an example of this conversion when $\beta$ is given by the braid word $\sigma_1\sigma_2^{-1}\sigma_1^{-1}\sigma_2$.}
\label{fig:braidleg}
\end{center}
\end{figure}

Let $Z$ denote the oriented $z$-axis, and consider the open book decomposition $(Z,\pi)$ of $\mathbb{R}^3$, where $\pi:\mathbb{R}^3\setminus Z\rightarrow S^1$ is the map given by \[\pi(z,r,\theta) = \theta.\] Since the open book $(Z,\pi)$ is compatible with $(\mathbb{R}^3,\xi_{rot}),$ it gives rise to an open book $\mathfrak{ob}$ compatible with $(S^3,\xi_{std})$ via the contactomorphism $\phi$ and the identification of $S^3$ with $\mathbb{R}^3 \cup \{\infty\}$. The binding of $\mathfrak{ob}$ is an unknot, $U=\phi(Z) \cup \{\infty\}$, and its pages are disks. Note that $\phi(T_\beta)$ is braided with respect to $\mathfrak{ob}$ since $T_{\beta}$ was braided with respect to $(Z,\pi)$. The leftmost diagram in Figure \ref{fig:braidaxis} shows the link $\phi(T_{\beta})$ together with its braid axis $U$.

In a slight abuse of notation, we will denote transverse link $\phi(T_{\beta})$ simply by $K$. The remainder of this section is devoted to studying the filtration induced by $-U\subset -S^3$ on the knot Floer chain complex for $K\subset -S^3$, and the relationship between this filtration and the invariant $\theta(K)=\lambda(L_{\beta})$. We start with a special grid diagram $G$ for $K\cup -U\subset S^3$ such that the front projection specified by $G$ is isotopic through front projections to the front projection shown on the right in Figure \ref{fig:braidaxis}. 

\begin{figure}[!htbp]
\labellist 
\hair 2pt 
\small

\pinlabel $\beta$ at 50 66
\pinlabel $\phi(T_\beta)$ at 183 14
\pinlabel $U$ at -6 14

\pinlabel $\beta$ at 361 68
\pinlabel $L_\beta$ at 430 17
\pinlabel $-\mathcal{U}$ at 270 17

\endlabellist 
\begin{center}
\includegraphics[height = 1.9cm]{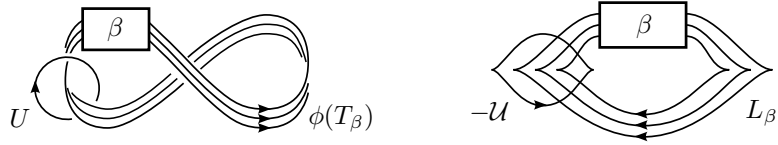}
\caption{\quad On the left, the link $\phi(T_{\beta})\cup U$. On the right, a Legendrian front projection for $L_{\beta}\cup -\mathcal{U}$, where $-\mathcal{U}$ is the $tb=-1$ representative of $-U$. This Legendrian link is smoothly isotopic to $K\cup -U$.}
\label{fig:braidaxis}
\end{center}
\end{figure}

Let $(T^2,\alphas,\betas,\zs_K\cup \zs_{-U},\ws_K\cup\ws_{-U})$ be the multi-pointed Heegaard diagram associated to $G$, with vertical circles labeled $\alpha_1,\dots,\alpha_k$ from left to right and horizontal circles labeled $\beta_1,\dots,\beta_k$ from bottom to top. We require that $\ws_{-U} = \{w_1,w_2\}$, where $w_1$ lies between $\alpha_1$ and $\alpha_2$ and between $\beta_1$ and $\beta_2$, and $w_2$ lies between some $\alpha_m$ and $\alpha_{m+1}$ and between $\beta_m$ and $\beta_{m+1}$. This condition on $\ws_{-U}$ implies that $-U$ divides $T^2$ into four rectangular regions, $R_1,R_2,R_3,R_4$, with corners at the four points of $\zs_{-U}\cup \ws_{-U}$. Here, $R_1$ is the square region bounded by $-U$, $R_2$ is the square region diagonal to $R_1$ and $R_3,R_4$ are the rectangular regions above and to the right of $R_1$. We require that all crossings of $K$ are contained within $R_1\cup R_2$. See Figure \ref{fig:examplegrid} for an example of (the Heegaard diagram associated to) a grid diagram $G$ with these properties.

\begin{figure}[!htbp]
\labellist 
\hair 2pt 
\small
\pinlabel $z$ at 297 297
\pinlabel $z$ at 274 274
\pinlabel $z$ at 251 251
\pinlabel $z$ at 181 204
\pinlabel $z$ at 204 181
\pinlabel $z$ at 135 227
\pinlabel $z$ at 228 134
\pinlabel $z$ at 158 157

\pinlabel $z$ at 88 87
\pinlabel $z$ at 65 64
\pinlabel $z$ at 41 41
\pinlabel $z$ at 19 110
\pinlabel $z$ at 112 18
\pinlabel $w_1$ at 19 17
\pinlabel $w_2$ at 112 109

\pinlabel $w$ at 88 133
\pinlabel $w$ at 228 204
\pinlabel $w$ at 181 227
\pinlabel $w$ at 135 297
\pinlabel $w$ at 298 41
\pinlabel $w$ at 41 181
\pinlabel $w$ at 204 274
\pinlabel $w$ at 275 64
\pinlabel $w$ at 65 158
\pinlabel $w$ at 158 251
\pinlabel $w$ at 252 87
\endlabellist 
\begin{center}
\includegraphics[height = 4.7cm]{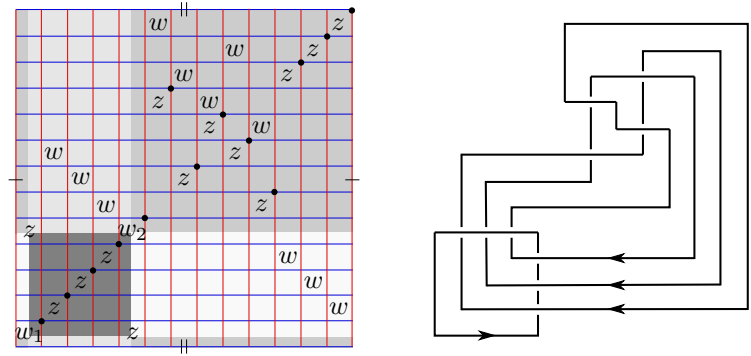}
\caption{\quad The Heegaard diagram associated to a special grid diagram for $K\cup -U$ in the case that $\beta$ is given by the braid word $\sigma_1\sigma_2^{-1}\sigma_1^{-1}\sigma_2$. The shaded regions are $R_1,\dots,R_4$ from darkest to lightest. The black dots indicate the components of the generator $\x_4$. }
\label{fig:examplegrid}
\end{center}
\end{figure}

Note that $(-T^2,\alphas,\betas,\zs_K\cup \zs_{-U},\ws_K\cup\ws_{-U})$ is a multi-pointed Heegaard diagram for the link $K\cup -U \subset -S^3$. Likewise, $\HD_4=(-T^2,\alphas,\betas,\zs_K,\ws_K\cup\ws_{-U})$ is a multi-pointed Heegaard diagram for $K\subset -S^3$ with free basepoints $w_1,w_2\in\ws_{-U}$.\footnote{We defined Heegaard diagrams $\HD_1,\HD_2,\HD_3$ in the previous section; hence, our notation $\HD_4$ for this diagram.} As discussed in Section \ref{sec:braid_char}, the Alexander grading $A_{-U}$ induces a filtration \begin{equation}\label{eqn:filt}\emptyset = \mscr{F}^{-U}_m(\HD_4)\subset\mscr{F}^{-U}_{m+1}(\HD_4)\subset\dots\subset\mscr{F}^{-U}_n(\HD_4)=\rm{CFK}^{-,2}(\HD_4),\end{equation} where $\mscr{F}^{-U}_{k}(\HD_4)$ is generated by $\{\x\in \mathbb{T}_{\alpha}\cup \mathbb{T}_{\beta}\,|\,A_{-U}(\x)\leq k\}.$ It is not hard to see that $A_{-U}(\mathbf{x})$ is equal, up to an overall shift, to the sum of the winding numbers of $-U$ around the components of $\bf x$ (c.f. \cite{mos}). It follows that generators in the bottommost nontrivial filtration level $\mscr{F}^{-U}_{bot}(\HD_4)$ are precisely those whose components are contained within the regions $R_1\cup R_2$.

Let $\x_4$ denote the generator of $\rm{CFK}^{-,2}(\HD_4)$ consisting of the intersection points at the upper right-hand corners of the squares containing the basepoints in $\zs_K\cup\ws_{-U}$. The components of $\x_4$ are contained in $R_1\cup R_2$, which implies that $\x_4\in \mscr{F}^{-U}_{bot}(\HD_4)$. Moreover, $\x_4$ is a cycle in $\rm{CFK}^{-,2}(\HD_4)$. We show in Proposition \ref{prop:bottomfiltration} that $[\x_4]$ generates the summand $H_{top}(\mscr{F}^{-U}_{bot}(\HD_4))$ of $H_*(\mscr{F}^{-U}_{bot}(\HD_4))$ in the top Maslov grading. Below, we describe in two steps the relationship between the $\x_4$ and the transverse invariant $\theta(K)$. 

Let $\HD_5 = (-T^2,\alphas',\betas',\zs_K\cup \zs_{-U},\ws_K\cup\ws_{-U})$ be the Heegaard diagram obtained from $\HD_4$ by handlesliding $\alpha_{2}$, $\beta_2$, $\alpha_{n+2}$ and $\beta_{n+2}$ over $\alpha_1$, $\beta_2$, $\alpha_{n+3}$ and $\beta_{n+3}$, respectively, and then isotoping so that the new curves $\alpha_2'$, $\beta_2'$ form a small configuration around $w_1$ and $\alpha_{n+2}'$, $\beta_{n+2}'$ form a small configuration around $w_2$. We will assume that $\alpha_i'$ and $\alpha_i$ are related by a small Hamiltonian isotopy for $i\neq 2,n+2$, and likewise for $\beta_i'$ and $\beta_i$. Let $\x_5$ be the generator of $\rm{CFK}^{-,2}(\HD_5)$ obtained from $\x_4$ by replacing the components of $\x_4$ on $\alpha_2\cap\beta_2$ and $\alpha_{n+2}\cap \beta_{n+2}$ by the intersection points on $\alpha_2'\cap\beta_2'$ and $\alpha_{n+2}'\cap \beta_{n+2}'$ with the smallest Maslov grading contributions (see Figure \ref{fig:H5} for an example), and replacing all other components of $\x_4$ by the corresponding intersection points between the $\alphas'$ and $\betas'$ curves.

\begin{figure}[!htbp]
\labellist 
\hair 2pt 
\small
\pinlabel $z$ at 346 343
\pinlabel $z$ at 323 320
\pinlabel $z$ at 299 296
\pinlabel $z$ at 229 249
\pinlabel $z$ at 252 226
\pinlabel $z$ at 181 271
\pinlabel $z$ at 277 179
\pinlabel $z$ at 205 202

\pinlabel $w_2$ at 141 138

\pinlabel $z$ at 113 112
\pinlabel $z$ at 89 87
\pinlabel $z$ at 65 63

\pinlabel $w_1$ at 39 36

\pinlabel $w$ at 114 178
\pinlabel $w$ at 276 249
\pinlabel $w$ at 230 272
\pinlabel $w$ at 183 342
\pinlabel $w$ at 346 64
\pinlabel $w$ at 64 226
\pinlabel $w$ at 253 319
\pinlabel $w$ at 323 87
\pinlabel $w$ at 89 202
\pinlabel $w$ at 206 296
\pinlabel $w$ at 300 110
\endlabellist 
\begin{center}
\includegraphics[height = 5.1cm]{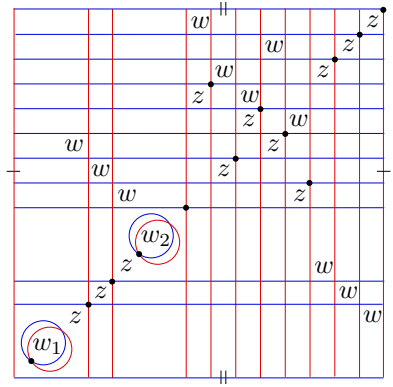}
\caption{\quad The Heegaard diagram $\HD_5$ for the example in Figure \ref{fig:examplegrid}. The black dots represent the generator $\x_5$. }
\label{fig:H5}
\end{center}
\end{figure}

\begin{proposition}
\label{prop:hslideisot}
The isomorphism \[F_{4,5}:\HFK^{-,2}(\HD_4)\rightarrow\HFK^{-,2}(\HD_5)\] associated to the above sequence of handleslides and isotopies sends $[\x_4]$ to $[\x_5]$.
\end{proposition}

\begin{proof}[Proof of Proposition \ref{prop:hslideisot}]
Let $\HD_{4.5}=(-T^2,\alphas,\betas',\zs_K,\ws_K\cup\ws_{-U})$.
The map $F_{4,5}$ is induced by the composition $f_{\alpha',\alpha,\beta'}\circ f_{\alpha,\beta,\beta'}$, where \begin{align*}
f_{\alpha,\beta,\beta'}&: \CFK^{-,2}(\HD_{4})\rightarrow \CFK^{-,2}(\HD_{4.5}),\\
f_{\alpha',\alpha,\beta'}&: \CFK^{-,2}(\HD_{4.5})\rightarrow \CFK^{-,2}(\HD_{5})
\end{align*}
are the pseudo-holomorphic triangle-counting maps associated to the multi-pointed Heegaard triple diagrams $(-T^2,\alphas,\betas,\betas',\zs_K,\ws_K\cup\ws_{-U})$ and $(-T^2,\alphas',\alphas,\betas',\zs_K,\ws_K\cup\ws_{-U})$, respectively. Let $\x_{4.5}$ be the generator of $\CFK^{-,2}(\HD_{4.5})$ obtained from $\x_4$ by replacing the components of $\x_4$ on $\alpha_2\cap\beta_2$ and $\alpha_{n+2}\cap \beta_{n+2}$ by the intersection points on $\alpha_2\cap\beta_2'$ and $\alpha_{n+2}\cap \beta_{n+2}'$ with the smallest Maslov grading contributions, and replacing all other components of $\x_4$ by the corresponding intersection points between the $\alphas$ and $\betas'$ curves. It is not hard to show that $f_{\alpha,\beta,\beta'}$ sends $\x_4$ to $\x_{4.5}$ and $f_{\alpha',\alpha,\beta'}$ sends $\x_{4.5}$ to $\x_5$. Below, we illustrate the proof of the first statement (our proof mimics that of \cite[Lemma 3.4]{Ver}); the second follows by a similar argument. These two statements prove Proposition \ref{prop:hslideisot}.

Let $\Theta$ denote the generator of the complex $\CFK^{-,2}(-T^2,\betas,\betas',\zs_K,\ws_K\cup\ws_{-U})$ in its top Maslov grading, and let $u_0\in \pi_2(\x_4,\Theta,\x_{4.5})$ denote the Whitney triangle whose domain $D(u_0)$ is the union of the small triangles, as shown in the example in Figure \ref{fig:H45}. Observe that $u_0$ admits a unique pseudo-holomorphic representative. Now, suppose $\y$ is another generator of $\CFK^{-,2}(\HD_{4.5})$, and let $u\in \pi_2(\x_4,\Theta,\y)$ be a Whitney triangle which avoids the basepoints in $\zs_{K}\cup\ws_{-U}$. The boundary of $D(u)-D(u_0)$ then consists of arcs along the $\alphas$ and $\betas'$ curves together with complete $\betas$ curves. Note that for every $\beta\in \betas$ there exists a periodic domain of $(-T^2,\betas,\betas',\zs_K,\ws_K\cup\ws_{-U})$ whose boundary is a sum of $\beta$ with curves in $\betas'$. Hence, there is a (unique) periodic domain $P_0$ of $(-T^2,\betas,\betas',\zs_K,\ws_K\cup\ws_{-U})$ for which the boundary of $D=D(u)-D(u_0)-P_0$ consists only of arcs along the $\alphas$ and $\betas'$ curves. 

\begin{figure}[!htbp]
\labellist 
\hair 2pt 
\small
\pinlabel $z$ at 345 342
\pinlabel $z$ at 321 318
\pinlabel $z$ at 297 294
\pinlabel $z$ at 228 247
\pinlabel $z$ at 252 224
\pinlabel $z$ at 181 270
\pinlabel $z$ at 276 178
\pinlabel $z$ at 205 201

\pinlabel $w_2$ at 147 142

\pinlabel $z$ at 113 110
\pinlabel $z$ at 89 85
\pinlabel $z$ at 65 62

\pinlabel $w_1$ at 31 31

\pinlabel $w$ at 114 177
\pinlabel $w$ at 275 246
\pinlabel $w$ at 229 270
\pinlabel $w$ at 182 340
\pinlabel $w$ at 345 63
\pinlabel $w$ at 64 224
\pinlabel $w$ at 252 317
\pinlabel $w$ at 323 84
\pinlabel $w$ at 89 200
\pinlabel $w$ at 205 294
\pinlabel $w$ at 299 108
\endlabellist 
\begin{center}
\includegraphics[height = 8cm]{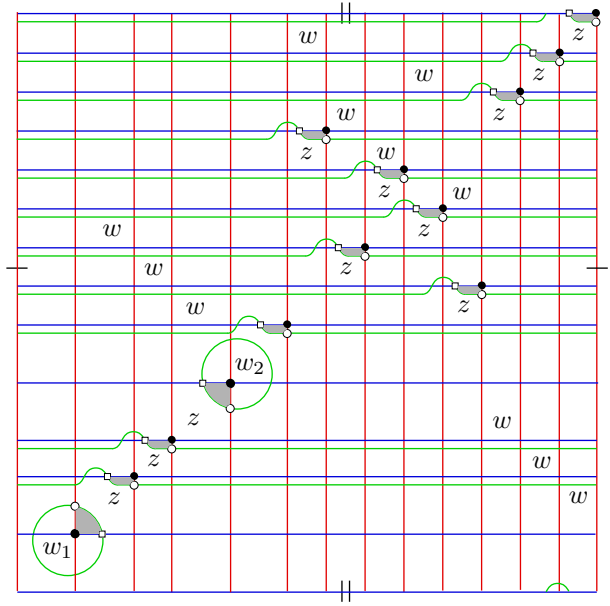}
\caption{\quad The  Heegaard triple diagram $(-T^2,\alphas,\betas,\betas',\zs_K,\ws_K\cup\ws_{-U})$ for the example in Figure \ref{fig:examplegrid}. The black dots represent the generator $\x_5$, the white dots represent $\x_{4.5}$ and the squares represent $\Theta$. The union of the small shaded triangles represents $u_0\in\pi_2(\x_4,\Theta,\x_{4.5})$.}
\label{fig:H45}
\end{center}
\end{figure}

It follows that $D$ is a domain of $(-T^2,\alphas,\betas',\zs_K,\ws_K\cup\ws_{-U})$ connecting $\x_{4.5}$ to $\y$. Note also that $D$ avoids the basepoints in $\zs_{K}\cup\ws_{-U}$. It is clear that any such domain which is not identically zero must have some negative multiplicities. Observe that $D(u_0)+P$ does not fully cover any region of $\Sigma-\alphas-\betas'$ for any periodic domain $P$ of $(-T^2,\betas,\betas',\zs_K,\ws_K\cup\ws_{-U})$. Therefore, if $u\neq u_0$, then $D$ is not identically zero. In this case, $D(u)$ must have some negative multiplicities as well (again, since $P_0$ does not fully cover any region of $\Sigma-\alphas-\betas'$), which implies that $u$ does not admit a holomorphic representative. In summary, the only pseudo-holomorphic triangle which contributes to $f_{\alpha,\beta,\beta'}(\x_4)$ is the unique pseudo-holomorphic representative of $u_0$. This implies that $f_{\alpha,\beta,\beta'}(\x_4)=\x_{4.5}.$
\end{proof}

Let $\HD_6$ be the diagram obtained from $\HD_5$ by performing two free index 0/3 destabilizations to remove the basepoints $w_1,w_2$. By construction, $\HD_6$ is the Heegaard diagram corresponding to a grid diagram whose associated front projection is isotopic to the front projection of $L_{\beta}$ shown in Figure \ref{fig:griddiagram}. Let $\x_6$ denote the generator of $\CFKm(\HD_6)$ consisting of intersection points at the upper right-hand corners of the squares containing the basepoints in $\zs_K$. Recall that $\theta(K):=[\x_6]\in \HFKm(\HD_6)$. It is clear that the projection and inclusion maps \begin{align*}j_{5,6}^2&:\CFK^{-,2}(\HD_5)\rightarrow \CFKm(\HD_6),\\
i^2_{6,5}&:\CFKm(\HD_6)\rightarrow \CFK^{-,2}(\HD_5),\end{align*} defined in Subsection \ref{sub:hfk}, send $\x_5$ to $\x_6$ and $\x_6$ to $\x_5$, respectively. Combining this with Proposition \ref{prop:hslideisot}, we have the following.

\begin{proposition}
\label{prop:transversedestab}
The compositions \begin{align*}(j^2_{5,6})_*\circ F_{4,5}&:\HFK^{-,2}(\HD_4)\rightarrow \HFKm(\HD_6),\\
F_{4,5}^{-1}\circ(i^2_{6,5})_*&:\HFKm(\HD_6)\rightarrow \HFK^{-,2}(\HD_4),\end{align*} send $[\x_4]$ to $\theta(K)$ and $\theta(K)$ to $[\x_4]$, respectively.
\end{proposition}

Since the class $\theta(K)$ is always nonzero \cite{OST}, the same is true of $[\x_4]$. Recall from Subsection \ref{sub:hfk} that the maps $i^2_{6,5}$ and $j^2_{5,6}$ increase Maslov grading by $2$. Therefore, since $M(\theta(K)) = sl(K)+1$, the Maslov grading of $\x_4$ is given by $M(\x_4) = sl(K)-1$. We conclude with the following characterization of $[\x_4]$.

\begin{proposition}
\label{prop:bottomfiltration} Let $b = \min\{j\,|\,H_*(\mscr{F}^{-U}_j(\HD_4)) \neq 0\},$ and let $H_t(\mscr{F}^{-U}_{b}(\HD_4))$ denote the summand of $H_*(\mscr{F}^{-U}_{b}(\HD_4))$ in the top Maslov grading. Then $H_t(\mscr{F}^{-U}_{b}(\HD_4))$ has rank one and is generated by $[\x_4].$

\end{proposition}

\begin{proof}[Proof of Proposition \ref{prop:bottomfiltration}]
Since $H_*(\mscr{F}^{-U}_{bot}(\HD_4))$ is nontrivial (it contains the nonzero class $[\x_4]$), it follows that $b=bot$ and $t=top$. Since the filtered quasi-isomorphism type of $\mscr{F}^{-U}(\HD_4)$ is an invariant of the link $K\cup -U \subset -S^3$, we know from the analogous result in the previous section that $H_{top}(\mscr{F}^{-U}_{bot}(\HD_4))$ has rank one. Moreover, we proved in that section that $t=top = sl(K)-1$. Therefore, $H_{top}(\mscr{F}^{-U}_{bot}(\HD_4))$ is generated by $[\x_4]$.
\end{proof}

%%%%%%%%%%%%%%%%%%%%%%%%%%%%%%%%%%%%%%%%%%%%%%%%%%%%%%%
\section{BRAID = GRID} % (fold)
\label{sec:braidgrid}
%%%%%%%%%%%%%%%%%%%%%%%%%%%%%%%%%%%%%%%%%%%%%%%%%%%%%%%

In this short section, we make precise the correspondence between our BRAID invariant $t$ and the GRID invariant $\theta$, using the results of Sections \ref{sec:braid_char} and \ref{sec:comb_char}. 

Suppose $K$ is a transverse knot in $(S^3,\xi_{std})$ which is braided with respect to the standard disk open book decomposition $(U,\pi)$ of $S^3$. Let $\HD_1,\dots,\HD_6$ and $\x_1,\dots,\x_6$ be the Heegaard diagrams and generators defined in Sections \ref{sec:braid_char} and \ref{sec:comb_char}. If we include the $\zs_U$ and $\zs_{-U}$ basepoints in the Heegaard diagrams $\HD_3$ and $\HD_4$, respectively, then each encodes the link $K\cup -U\subset -S^3$. It follows that $\HD_4$ may be obtained from $\HD_3$ by a sequence of isotopies and handleslides avoiding all basepoints, together with index 1/2 (de)stabilizations and linked index 0/3 (de)stabilizations involving only the basepoints in $\zs_K\cup\ws_K$.  The sets $\zs_K$, $\ws_K$, $\zs_U$ and $\ws_U$ are identified, respectively, with $\ws_K$, $\zs_K$, $\zs_{-U}$ and $\ws_{-U}$ via these moves. Associated to this sequence of Heegaard moves is a chain map \[f_{3,4}: \CFK^{-,2}(\HD_3) \rightarrow \CFK^{-,2}(\HD_4)\] which induces an isomorphism \[F_{3,4}:\HFK^{-,2}(\HD_3) \rightarrow \HFK^{-,2}(\HD_4).\] Since $f_{3,4}$ respects the filtrations of $\CFK^{-,2}(\HD_3)$ and $\CFK^{-,2}(\HD_4)$ induced by $A_{-U}$, it follows from Propositions \ref{prop:bottomfiltrationa} and \ref{prop:bottomfiltration} that $F_{3,4}$ sends $[\x_3]$ to $[\x_4]$.

Recall from Subsection \ref{sub:hfk} that the inclusion map \[(i^2_{1,2})_*:\HFKm(\HD_1)\rightarrow\HFK^{-,2}(\HD_2)\] from Section \ref{sec:braid_char} induces an isomorphism from $\HFKm(\HD_1)$ to the summand $(\cap_{i=1}^2 \coker \psi_{w_i})[2]$ of $\HFK^{-,2}(\HD_2)[2],$ where $w_1,w_2$ are the basepoints in $\ws_{-U}$. Since the basepoints actions $\psi_{w_1},\psi_{w_2}$ commute with the maps associated to Heegaard moves, the composition \[F_{4,5}\circ F_{3,4}\circ F_{3,2}^{-1}:\HFK^{-,2}(\HD_2)\rightarrow \HFK^{-,2}(\HD_5)\] restricts to an isomorphism from this summand to the analogous summand of $\HFK^{-,2}(\HD_5)[2]$. Finally, the projection map \[(j^2_{5,6})_*:\HFK^{-,2}(\HD_5)\rightarrow\HFK^{-}(\HD_6)\] defined in Section \ref{sec:comb_char} restricts to an isomorphism from the summand $(\cap_{i=1}^2 \coker \psi_{w_i})[2]$ of $\HFK^{-,2}(\HD_5)[2]$ to $\HFKm(\HD_6)$. This proves that the composition \[(j^2_{5,6})_*\circ F_{4,5}\circ F_{3,4}\circ F_{3,2}^{-1}\circ(i^2_{1,2})_*:\HFK^{-}(\HD_1)\rightarrow \HFK^{-}(\HD_6)\] is an isomorphism. By Propositions \ref{prop:transversedestaba} and \ref{prop:transversedestab} and the discussion about $F_{3,4}$ above, this composition also sends $t(K):=[\x_1]$ to $\theta(K):=[\x_6]$. In other words, we have shown the following.

\begin{theorem}
\label{thm:braid=grid}
Let $K$ be a transverse knot in $(S^3,\xi_{std})$. There exists an isomorphism of bigraded $\F[U]-$modules, \[\psi:\HFKm(-S^3,K)\rightarrow\HFKm(-S^3,K),\] which sends $t(K)$ to $\theta(K)$.\footnote{An analogous statement holds for transverse links in $(S^3,\xi_{std})$.}
\end{theorem}

Combined with Theorem \ref{thm:LOSS=t}, this completes the proof of Theorem \ref{thm:LOSS_Grid_trans}. \qed

\bibliographystyle{myalpha}
\nocite{*}
\bibliography{References}

\end{document}